\newcommand{\Set}{{\rm Set}}

\newcommand{\Sh}{{\rm Sh}}
\newcommand{\Top}{{\rm Top}}

\newcommand{\Space}{{\rm Space}}
\newcommand{\Diff}{{\rm Diff}}

\newcommand{\Chen}{{\sf Chen}}
\newcommand{\Simplicial}{{\sf F}}
\newcommand{\Diffeological}{{\sf Diffeological}}
\newcommand{\D}{{\sf D}}

\newcommand{\Ch}{{\rm Ch}}
\newcommand{\So}{{\rm So}}

\newcommand{\im}{{\rm im}}
\renewcommand{\hom}{{\rm hom}}
\newcommand{\To}{\Rightarrow}
\renewcommand{\to}{\rightarrow}

\newcommand{\maps}{\colon}

\newcommand{\Conc}{\mathrm{Conc}}
\newcommand{\X}{{\mathbf{X}}}
\newcommand{\Y}{{\mathbf{Y}}}
\newcommand{\f}{{\mathbf{f}}}

\newcommand{\op}{{\rm op}}

\newcommand{\ev}{{\rm ev}}

\newcommand{\R}{{\mathbb R}}

\newcommand{\calc}{\ensuremath{\mathcal{C}}\xspace}
\newcommand{\cald}{\ensuremath{\mathcal{D}}\xspace}

\newcommand{\cinf}{\ensuremath{\calc^\infty}\xspace}

\documentclass{article}
\usepackage{amsfonts,amsmath,amssymb,amsthm,amscd,mathrsfs,xspace,multicol}
\usepackage{latexsym}
\usepackage[all]{xy}
\usepackage[breaklinks=true]{hyperref}

\newcommand{\comment}[1]{}
\renewcommand{\u}[1]{\underline{#1}}

\newtheorem{theorem}{Theorem}
\newtheorem{definition}[theorem]{Definition}
\newtheorem{lemma}[theorem]{Lemma}

\newtheorem{proposition}[theorem]{Proposition}

\newdir{ >}{{}*!/-7pt/@{>}}
\newdir{> >}{*!/0pt/@{>}*!/-5pt/@{>}}

\title{
Convenient Categories of Smooth Spaces}

\author{John C.\ Baez and Alexander E.\ Hoffnung \\
Department of Mathematics,  University of California\\
Riverside, California 92521 \\
USA \\
\\
email: baez@math.ucr.edu, alex@math.ucr.edu }

\date{July 10, 2008}
\begin{document}
\bibliographystyle{plain}
\maketitle

\begin{abstract}
\noindent 
A `Chen space' is a set $X$ equipped with a collection of
`plots' --- maps from convex sets to $X$ --- satisfying 
three simple axioms.  While an 
individual Chen space can be much worse than a smooth manifold,
the category of all Chen spaces is much better behaved
than the category of smooth manifolds.  For example, any subspace 
or quotient space of a Chen space is a Chen space, and the space of 
smooth maps between Chen spaces is again a Chen space.
Souriau's `diffeological spaces' share these convenient properties.  
Here we give a unified treatment 
of both formalisms.  Following ideas of Penon and Dubuc, we show that 
Chen spaces, diffeological spaces, and even simplicial complexes are 
examples of `concrete sheaves on a concrete site'.  As a result, 
the categories of such spaces are locally cartesian closed, with
all limits, all colimits, and a weak subobject classifier.  
For the benefit of differential geometers, our treatment explains 
most of the category theory we use.

\end{abstract}

\section{Introduction}
\label{introduction}

Algebraic topologists have become accustomed to working in a category
of spaces for which many standard constructions have good formal
properties: mapping spaces, subspaces and quotient spaces,
limits and colimits, and so on.  In differential geometry the 
situation is quite different, since the most popular category,
that of finite-dimensional smooth manifolds, lacks almost all
these features.  So, researchers are beginning to seek a `convenient 
category' of smooth spaces in which to do differential geometry.

In this paper we study two candidates: Chen spaces and diffeological
spaces.  But before we start, it is worth recalling the lesson of 
algebraic topology in a bit more detail.  Dissatisfaction arose when it 
became clear that the category of topological spaces suffers from a 
defect: there is generally no way to give the set $C(X,Y)$ of 
continuous maps from a space $X$ to a space $Y$ a topology such 
that the natural map 
\[      
\begin{array}{rcl}
      C(X \times Y, Z) &\to& C(X, C(Y, Z))  \\
                   f &\mapsto& \tilde{f} 
\end{array}
\]
\[              \tilde f(x)(y) = f(x,y)  \]
is a homeomorphism.   In other words, this category fails to
be cartesian closed.  This led to the search for a better framework ---
or as Brown \cite{Brown:1963} put it, a ``convenient category".  

Steenrod's paper ``A convenient category of topological spaces" 
\cite{Steenrod:1967} popularized the idea of restricting attention 
to spaces with a certain property to obtain a cartesian closed category.  
It was later realized that by adjusting this property a bit, we can 
also make quotient spaces better behaved.  The resulting category --- with 
compactly generated spaces as objects, and continuous maps as 
morphisms --- has now been widely adopted in algebraic topology 
\cite{May}.  This shows that it is perfectly possible, and at times quite
essential, for a discipline to change the category that constitutes
its main subject of inquiry.  Something similar happened in
algebraic geometry when Grothendieck invented schemes as a
generalization of algebraic varieties.

Now consider differential geometry.  Like the category of topological 
spaces, the category of smooth manifolds fails to be cartesian closed.  
Indeed, if $X$ and $Y$ are finite-dimensional smooth manifolds, the 
space of smooth maps $\cinf(X,Y)$ is hardly ever the same sort
of thing.  It is a kind of {\it infinite-dimensional} manifold --- but 
making the space of smooth maps between \textit{these} into an
infinite-dimensional manifold becomes more difficult.   
It can be done \cite{Kriegl:1984,Michor:1984}, but
there are still many spaces on which we can do
differential geometry that do not live in the resulting 
cartesian closed category.  The simplest examples are manifolds
with boundary, or more generally manifolds with corners.
There are also many formal properties one might want, which are 
lacking: for example, a subspace or quotient space of a manifold is 
rarely a manifold, and the category of manifolds does not have limits 
and colimits.

In 1977, Chen defined a simple notion that avoids all these problems
\cite{Chen:1977}.  A `Chen space' is a set $X$ equipped with a
collection of `plots' --- maps $\varphi \maps C \to X$ where $C$ is
any convex subset of any Euclidean space $\R^n$ --- obeying three
simple axioms.  Despite a superficial resemblance to charts in the
theory of manifolds, plots are very different: we should think of a
plot in $X$ as an {\it arbitrary} smooth map to $X$ from a convex
subset of a Euclidean space of {\it arbitrary} dimension.  So instead
of ensuring that Chen spaces look nice locally, plots play a different
role: they determine which maps between Chen spaces are smooth.  Given
a map $f \maps X \to Y$ between Chen spaces, $f$ is `smooth' if and
only if for any plot in $X$, say $\varphi \maps C \to X$, the
composite $f \varphi \maps C \to Y$ is a plot in $Y$.

In 1980, Souriau introduced another category of smooth spaces:
`diffeological spaces' \cite{Souriau:1980}.  The definition of these
closely resembles that of Chen spaces: the only difference is that the
domain of a plot can be any {\it open} subset of $\mathbb{R}^n$,
instead of any convex subset.  As a result, Chen spaces and
diffeological spaces have many similar properties.  So, in what
follows, we use `smooth space' to mean {\it either} Chen space {\it
or} diffeological space.  We shall see that:
\begin{itemize}
\item
Every smooth manifold is a smooth space, and a map between smooth
manifolds is smooth in the new sense if and only if it is smooth
in the usual sense.
\item 
Every smooth space has a natural topology, and smooth maps 
between smooth spaces are automatically continuous.
\item 
Any subset of a smooth space becomes a smooth space in a natural
way, and the inclusion of this subspace is a smooth map.
Subspaces of a smooth space are classified by their characteristic
functions, which are smooth maps taking values in 
$\{0,1\}$ equipped with its `indiscrete' smooth structure.
So, we say $\{0,1\}$ with its indiscrete smooth
structure is a `weak subobject classifier' for the category
of smooth spaces (see Def.\ \ref{subobject.class}).
\item
The quotient of a smooth space under any equivalence relation becomes
a smooth space in a natural way, and the quotient map is smooth.
\item
The category of smooth spaces has all limits and colimits.
\item
Given smooth spaces $X$ and $Y$, the set $\cinf(X,Y)$ of all
smooth maps from $X$ to $Y$ can be made into a smooth space in
such a way that the natural map 
\[
      \cinf(X \times Y, Z) \to \cinf(X, \cinf(Y, Z))  
\]
is a smooth map with a smooth inverse.
So, the category of smooth spaces is cartesian closed.

\item
More generally, given any smooth space $B$, the category of 
smooth spaces `over $B$' --- that is, equipped with
maps to $B$ --- is cartesian closed.  So, we say the category
of smooth spaces is `locally cartesian closed' (see Def.\ 
\ref{loc.cart.closed} for details).
\end{itemize}
The goal of this paper is to present a unified approach to Chen
spaces and diffeological spaces that explains why they share
these convenient properties.

All this convenience comes
with a price: both these categories contain many spaces 
whose local structure is far from that of Euclidean space.  
This should not be surprising.  For example, the subset of 
a manifold defined by an equation between smooth maps,
\[     Z = \{x \in M \colon f(x) = g(x) \}, \]
is not usually a manifold in its own right.  In fact, $Z$ can easily
be as bad as the Cantor set if $M = \R$.  But it is a smooth space.
It is nice having the solution set of an equation between smooth maps
be a smooth space, but the price we pay is that a smooth space can be
locally as bad as the Cantor set.

So, we should not expect the theory of smooth spaces to support the
wealth of fine-grained results familiar from the theory of smooth
manifolds.  Instead, it serves as a large context for general ideas.
For a taste of just how much can be done here, see Iglesias--Zemmour's
book on diffeological spaces \cite{Iglesias}.  There is no real 
conflict, since smooth manifolds form a full subcategory of the 
category of smooth spaces.  We can use the big category for abstract 
constructions, and the small one for theorems that rely on good 
control over local structure.

Since we want differential geometers to embrace the notions
we are describing, our treatment will be as self-contained
as possible.  This requires a little introduction to sheaves
on sites, because the key fact underlying our main results
is that both Chen spaces and diffeological spaces are examples
of `concrete sheaves on a concrete site'.   For example, Chen spaces 
are sheaves on a site $\Chen$: the category whose objects are convex 
subsets of $\mathbb{R}^n$ and whose morphisms are smooth maps, equipped 
with a certain Grothendieck topology.  However, not {\it all} sheaves 
on this site count as Chen spaces, but only those satisfying a certain 
`concreteness' property, which guarantees that any Chen space has a 
well-behaved underlying set.   Formulating this property uses the fact 
that $\Chen$ itself is a `concrete site'.  Similarly, the category 
of diffeological spaces can be seen as the category of concrete sheaves 
on a concrete site $\Diffeological$.

The category of {\it all} sheaves on a site is extremely nice: 
it is a topos.  Here, following ideas of Penon
\cite{Penon:1973, Penon:1977} and Dubuc \cite{Dubuc:1979,Dubuc2:2006},
we show that the category of concrete sheaves on a concrete site 
is also nice, but slightly less so: it is a `quasitopos'. 
This yields many of the good properties listed above.

Various other notions of `smooth space' are currently being studied.
Perhaps the most elegant approach is synthetic differential geometry 
\cite{Kock}, which drops the assumption that a smooth space be a set
equipped with extra structure.  This gives a topos of smooth spaces,
and it allows a rigorous treatment of calculus using infinitesimals.

Most other approaches treat smooth spaces as sets equipped with a
specified class of `maps in', `maps out', or `maps in and out'.  
We recommend Stacey's work \cite{Stacey:2007} for a detailed comparison
of these approaches.  Chen and Souriau take the `maps in' approach, 
where a plot in a smooth space $X$ is a map into $X$, and a function 
$f \maps X \to Y$ between smooth spaces is smooth when its composite 
with every plot in $X$ is a plot in $Y$.  Smith \cite{Smith:1966}, 
Sikorski \cite{Kreck,Sikorski:1971} and Mostow \cite{Mostow} instead 
follow the `maps out' approach, in which a smooth space $X$ comes equipped 
with a collection of `coplots' $\varphi \maps X \to C$ for certain spaces 
$C$, and  a map $f \maps X \to Y$ between smooth spaces is smooth when 
its composite with every coplot on $Y$ is a coplot on $X$.   Fr\"{o}licher 
takes the `maps in and out' approach, in which a smooth space is 
equipped with both plots and coplots \cite{Frolicher:1982,Laubinger}.  
This gives two ways to determine the smoothness of a map between 
smooth spaces, which are required to give the same answer.  Our work covers 
a wide class of definitions that take the `maps in' approach.

The structure of the paper is as follows.  In Section \ref{smooth_spaces}, 
we define Chen spaces and diffeological spaces and give some examples.  We 
also discuss the relation between these two formalisms, focusing on manifolds 
with corners and the work of Stacey \cite{Stacey:2007}.
In Section \ref{convenient_properties}, we list many convenient properties 
shared by these categories.  In Section \ref{generalized_spaces} we recall
the concept of a sheaf on a site and show that Chen spaces and 
diffeological spaces are `concrete' sheaves on `concrete' sites.  Simplicial 
complexes give another interesting example.  In Section \ref{quasitopos} we 
show that any category of concrete sheaves on a concrete site is a quasitopos 
with all limits and colimits.   Most of the properties described in Section 
\ref{convenient_properties} follow as a direct result.  

\section{Smooth Spaces}
\label{smooth_spaces}

Souriau's notion of a `diffeological space' \cite{Souriau:1980} is very simple:

\begin{definition}
An {\bf open set} is an open subset of $\R^n$.  A function 
$f\maps U\to U'$ between open sets is called {\bf smooth} if it has 
continuous derivatives of all orders.  
\end{definition}

\begin{definition}
A {\bf diffeological space} is a set $X$ equipped with, for each open 
set $U$, a set of functions
\[\varphi\maps U\to X\]
called {\bf plots in $X$}, such that:
\begin{enumerate}
\item  If $\varphi$ is a plot in $X$ and $f\maps U' \to U$ is
a smooth function between open sets, then $\varphi f$ is a plot in $X$.
\item Suppose the open sets $U_j \subseteq U$ form an open cover of the
open set $U$, with inclusions $i_j \maps U_j \to U$.  
If $\varphi i_j$ is a plot in $X$
for every $j$, then $\varphi$ is a plot in $X$. 
\item Every map from the one point of $\R^0$ to $X$ is a plot in $X$.
\end{enumerate}
\end{definition}

\begin{definition}
Given diffeological spaces $X$ and $Y$, a function $f\maps X\to Y$ is a 
{\bf smooth map} if, for every plot $\varphi$ in $X$, the composite 
$f\varphi$ is a plot in $Y$.
\end{definition}

Chen actually considered several different definitions.  Here we use
his final, most refined approach \cite{Chen:1977}, which closely
resembles Souriau's:

\begin{definition} \label{convex.set}
A {\bf convex set} is a convex subset of $\R^n$ with nonempty interior.  
A function $f \maps C \to C'$ between convex sets is called {\bf smooth} if 
it has continuous derivatives of all orders.
\end{definition}

\begin{definition}
A {\bf Chen space} is a set $X$ equipped with, for each convex set $C$, a 
set of functions
\[\varphi\maps C\to X\]
called {\bf plots in $X$}, satisfying these axioms:
\begin{enumerate}
\item  If $\varphi$ is a plot in $X$ and $f\maps C' \to C$ is
a smooth function between convex sets, then $\varphi f$ is a plot in $X$.
\item Suppose the convex sets $C_j \subseteq C$ form an open cover of the
convex set $C$ with its topology as a subspace of $\R^n$.  Denote the
inclusions as $i_j \maps C_j \to C$.  If $\varphi i_j$ is a plot in $X$
for every $j$, then $\varphi$ is a plot in $X$. 
\item Every map from the one point of $\R^0$ to $X$ is a plot in $X$.
\end{enumerate}
\end{definition}

\begin{definition}
Given Chen spaces $X$ and $Y$, a function $f\maps X\to Y$ is a 
{\bf smooth map} if, for every plot $\varphi$ in $X$, the composite 
$f\varphi$ is a plot in $Y$.
\end{definition}

It is instructive to see how Chen's definition evolved.
Of course he did not speak of `Chen spaces'; he called them 
`differentiable spaces'.  
In 1973, he took a differentiable space to be a Hausdorff space 
$X$ equipped with continuous plots $\varphi \maps C \to X$ satisfying
axioms 1 and 3 above, where the domains $C$ were closed convex subsets of 
Euclidean space \cite{Chen:1973}.  In 1975, he added a preliminary version 
of axiom 2 and dropped the condition that $X$ be Hausdorff \cite{Chen:1975}.
  
Starting in 1977, Chen used a definition equivalent to the one 
above \cite{Chen:1977,Chen:1982}.    In particular, he dropped the 
topology on $X$, the continuity of $\varphi$, and the condition that $C$ 
be closed.  This marks an important realization, emphasized by Stacey 
\cite{Stacey:2007}:  we can give a space a smooth structure without 
first giving it a topology.  Indeed, we shall see that a smooth structure 
determines a topology!  

The notion of a smooth function $f \maps C \to C'$ between convex sets 
is a bit subtle, particularly for points on the boundary of $C$.   
One tends to imagine $C$ as either open or 
closed, but the generic situation is more messy.  For example, $C$ could 
be the closed unit disk $D^2$ minus the set $Q$ of points on the unit circle 
with rational coordinates.  Both $Q$ and its complement are dense in the
unit circle.

Situations like this, while far from our main topic of interest,
deserve a little thought.  So, suppose $C \subseteq \R^n$ and $C' \subseteq 
\R^m$ are convex subsets with nonempty interior.  To define the $k$th 
derivative of a function from $C$ to $C'$ it suffices to define the first 
derivative of a function $F \maps C \to V$ for any finite-dimensional normed 
vector space $V$, since when this derivative exists it will be a function 
$dF$ from $C$ to the normed vector space of linear maps $\hom(\R^n,V)$.
We can then define the derivative of this function, and so on.  So: we say 
the derivative of $F$ exists at the point $x \in C$ if there is a linear 
map $(dF)_x \maps \R^n \to V$ such that 
\[         \frac{\|F(y) - F(x) - dF_x(y - x)\|}{\|y - x\|} \to 0 \]
as $y \to x$ for $y \in C - \{x\}$.   Note that since $C$ is convex with
nonempty interior, $dF_x$ is unique if it exists.

This is the usual definition going back to Fr\'echet, and scarcely worth 
remarking on, except for the obvious caveat that $y$ must lie in $C$.  
In the case $C = [0,1]$, this means we are using one-sided derivatives 
at the endpoints.  In the case of the convex set $D^2 - Q$, it means 
we are using a generalization of one-sided derivatives at all points 
on the boundary of this set, which is the unit circle minus $Q$.
 
Luckily, whenever $C$ and $C'$ are convex sets, we can characterize
smooth functions $f \maps C \to C'$ in three equivalent ways:
\begin{enumerate}
\item 
The function $f \maps C \to C'$ has continuous derivatives of all orders.
\item
The function $f \maps C \to C'$ 
has continuous derivatives of all orders in the interior of $C$, and these 
extend continuously to the boundary of $C$.
\item
If $\gamma \maps \R \to C$ is a smooth curve in $C$, then 
$f \gamma $ is a smooth curve in $C'$.
\end{enumerate}
The equivalence of conditions 1 and 2 is not hard; the equivalence
of 2 and 3 was proved by Kriegl \cite{Kriegl:1997}, and appears
as Theorem 24.5 in Kriegl and Michor's book \cite{KM}.

Since most of our results apply both to Chen spaces and diffeological
spaces, we lay down the following conventions:

\begin{definition}\label{convention}
We use {\bf smooth space} to mean either a Chen space or a diffeological 
space, and use $\cinf$ to mean either the category of Chen spaces and 
smooth maps, or diffeological spaces and smooth maps.   We use the term 
{\bf domain} to mean either a convex set or an open set, depending on 
context.

Henceforth, any statement about smooth spaces or the category $\cinf$ 
holds for \emph{both} Chen spaces and diffeological spaces.  
\end{definition}

\subsection{Examples}

Next we give some examples.  For these it is
handy to call the set of plots in a smooth space its {\bf smooth 
structure}.  So, we may speak of taking a set and putting a smooth 
structure on it to obtain a smooth space.

\begin{enumerate}
\item Any domain $D$ becomes a smooth space, where the plots 
$\varphi\maps D' \to D$ are just the smooth functions.

\item Any set $X$ has a {\bf discrete} smooth structure such that the
plots $\varphi \maps D \to X$ are just the constant functions.

\item Any set $X$ has an {\bf indiscrete} smooth structure where every
function $\varphi \maps D \to X$ is a plot.

\item Any smooth manifold $X$ becomes a smooth space where $\varphi \maps D
\to X$ is a plot if and only if $\varphi$ has continuous derivatives of
all orders.  Moreover, if $X$ and $Y$ are smooth manifolds, $f \maps 
X \to Y$ is a morphism in $\cinf$ if and only if it is 
smooth in the usual sense. 

\item Given any smooth space $X$, we can endow it with a new smooth 
structure where we keep only the plots of $X$ that factor through a chosen 
domain $D_0$.    When $D_0 = \R$ this smooth structure is called the 
`wire diffeology' in the theory of diffeological spaces
\cite{Iglesias}.  While this construction gives many examples 
of smooth spaces, these seem to be useful mainly as counterexamples to 
naive conjectures.

\item Any topological space $X$ can be made into a smooth space where we 
take the plots to be all the continuous maps $\varphi \maps C \to X$.  
Since every smooth map is continuous this defines a smooth structure.
Again, these examples mainly serve to disprove naive conjectures.
\end{enumerate}

If $\Diff$ is the category of smooth finite-dimensional manifolds and 
smooth maps, our fourth example above gives a full and faithful functor
\[          \Diff \to \cinf .\]
So, we can think of $\cinf$ as a kind of `extension' or `completion' 
of $\Diff$ with better formal properties.

Any smooth space $X$ can be made into topological space with the
finest topology such that all plots $\varphi \maps D \to X$ are
continuous.  With this topology, smooth maps between smooth spaces
are automatically continuous.  This gives a faithful functor
\[         \cinf \to \Top . \]
In particular, if we take a smooth manifold, regard it as a smooth
space, and then turn it into a topological space this way, we recover
its usual topology.

\subsection{Comparison}

We should also say a bit about how Chen spaces and diffeological 
spaces differ, and how they are related.  To begin with, let us
compare their treatment of manifolds with boundary, or more generally
manifolds with corners \cite{Janich,Laures}.  

An $n$-dimensional manifold with corners $M$ has charts of the form 
$\varphi \maps X_k \to M$, where 
\[    X_k = \{(x_1, \dots, x_n) \in \R^n \colon  x_1, \dots,  x_k \ge 0 \} \]
for $k = 0, \dots, n$.  The case $k = 1$ gives a half-space, familiar
from manifolds with boundary.  Since $X_k \subset \R^n$ 
is convex, any chart $\varphi \maps X_k \to M$ can be made 
into a plot in Chen's sense.   So, if we make $M$ 
into a Chen space where the plots $\varphi \maps C \to M$ are just 
maps that are smooth in the usual sense, it follows that any map between 
manifolds with corners $f \maps M \to N$ is smooth as a map of Chen
spaces if and only if it is smooth in the usual sense.  

However, the subset $X_k \subset \R^n$ is typically not open.  
So, we cannot make a chart for a manifold with corners into a plot
in the sense of diffeological spaces.  Nonetheless, we can make any 
manifold with corners $M$ into a diffeological space where the plots 
$\varphi \maps U \to M$ are the maps that are smooth in the usual 
sense.  And then, in fact, a map between manifolds with corners is 
smooth as a map between diffeological spaces if and only if it 
is smooth in the usual sense!

The key to seeing this is the theorem of Kriegl mentioned above.  
Since the issues involved are local, it suffices to consider maps $f \maps 
X_k \to \R^m$.  Suppose $f \maps X_k \to \R^m$ is smooth in the sense of 
diffeological spaces.  Then the composite $f \gamma$ is smooth for any 
smooth curve $\gamma \maps \R \to X_k$.  By Kriegl's theorem, 
this implies that $f$ has continuous derivatives of all orders in the 
interior of $X_k$, extending continuously to the boundary.  So, $f$ is 
smooth in the usual sense for manifolds with corners.  Conversely, any 
$f \maps X_k \to \R^m$ smooth in the usual sense is clearly 
smooth in the sense of diffeological spaces.   

Stacey has given a more general comparison of Chen spaces versus diffeological 
spaces \cite{Stacey:2007}.  To briefly summarize this, let us write 
$\Chen\Space$ for the category of Chen spaces, and $\Diffeological\Space$ 
for the category of diffeological spaces.   Stacey has shown that
these categories are not equivalent.  However, he has 
constructed some useful functors relating them.  These take advantage of 
the fact that every open subset of $\mathbb{R}^n$ becomes a Chen space 
with its subspace smooth structure, and conversely, every convex subset 
of $\mathbb{R}^n$ becomes a diffeological space.  

Using this, Stacey defines for any Chen space $X$ a diffeological space 
$\So X$ with the same underlying set, where $\varphi \maps U \to \So X$ 
is a plot if and only if  $\varphi \maps U \to X$ is a smooth map 
between Chen spaces.   This extends to a functor
\[             \So \maps \Chen\Space \to \Diffeological\Space  \]
that is the identity on maps.  He also defines for any diffeological
space $Y$ a Chen space $\Ch^\sharp Y$ with the same underlying set, where
$\varphi \maps C \to \Ch^\sharp Y$ is a plot if and only if $\varphi \maps
C \to Y$ is a smooth map between diffeological spaces.  Again, this extends 
to a functor
\[             \Ch^\sharp \maps \Diffeological\Space \to  \Chen\Space \]
that is the identity on maps.  Stacey shows that 
\begin{center}
  $f \maps X \to \Ch^\sharp Y$ is a smooth map between Chen spaces \break
    $\Updownarrow$ \break
  $f \maps \So X \to Y$ is a smooth map between diffeological spaces. 
\end{center}
In other words, $\Ch^\sharp$ is the right adjoint of $\So$.  

The functor $\So$ also has a left adjoint
\[             \Ch^\flat \maps \Diffeological\Space \to  \Chen\Space \]
which acts as the identity on maps.   
This time the adjointness means that
\begin{center}
     $f \maps \Ch^\flat Y \to X$ is a smooth map between Chen spaces  \break
     $\Updownarrow$ \break  
     $f \maps Y \to \So X$ is a smooth map between diffeological spaces.
\end{center}

Furthermore, Stacey shows that both these composites:
\[
\xymatrix{
  \Diffeological\Space \ar[r]^(.6){\Ch^\sharp} & 
  \Chen\Space \ar[r]^(.4){\So} & 
  \Diffeological\Space 
  }
\]
\[
\xymatrix{
  \Diffeological\Space \ar[r]^(.6){\Ch^\flat} & 
  \Chen\Space \ar[r]^(.4){\So} & 
  \Diffeological\Space 
  }
\]
are equal to the identity.   With a little work, it follows that both 
$\Ch^\sharp$ and $\Ch^\flat$ embed $\Diffeological\Space$ isomorphically 
as a full subcategory of $\Chen\Space$: a `reflective' subcategory in 
the first case, and a `coreflective' one in the second.

The embedding $\Ch^\flat$ is a bit strange: as shown by Stacey, even
the ordinary closed interval fails to lie in its image!  To see this,
he takes $I$ to be $[0,1] \subset \R$ made into a Chen space with its
subspace smooth structure.  If $I$ were isomorphic to a Chen space in
the image of $\Ch^\flat$, say $I \cong \Ch^\flat X$, we would then
have $\Ch^\flat \So I = \Ch^\flat \So \Ch^\flat X = \Ch^\flat X \cong
I$.  However, he shows explicitly that $\Ch^\flat \So I$ is not
isomorphic to $I$; it is the unit interval equipped with a nonstandard
smooth structure.

The embedding $\Ch^\sharp$ lacks this defect, since $\Ch^\sharp \So I = I$.  
For an example of a Chen space not in the image of $\Ch^\sharp$, 
we can resort to $\Ch^\flat \So I$.  Suppose there were a diffeological 
space $X$ with $\Ch^\sharp X \cong \Ch^\flat \So I$.  Then we would 
have $\So \Ch^\sharp X \cong \So \Ch^\flat \So I$, hence $X \cong \So I$.
But this is a contradiction, since we know that $\Ch^\sharp$ applied 
$\So I$ gives $I$, which is not isomorphic to $\Ch^\flat \So I$.

Luckily, the embedding $\Ch^\sharp$ works well for manifolds
with corners.  In particular, if $\Diff_c$ is the category of 
manifolds with corners and smooth maps, we have a commutative triangle 
\[
\xymatrix{
& \Diffeological\Space \ar[dd]^{\Ch^\sharp}  \\
\Diff_c \ar[ur] \ar[dr] \\
&  \Chen\Space 
}
\]
where the diagonal arrows are the full and faithful functors described
earlier.  

\section{Convenient Properties of Smooth Spaces}
\label{convenient_properties}

Now we present some useful properties shared by Chen spaces and 
diffeological spaces.  Following Def.\ \ref{convention}, we call either 
kind of space a `smooth space', and we use $\cinf$ to denote either 
the category of Chen spaces or the category of diffeological spaces.  
Most of the proofs are straightforward diagram chases, but we defer 
all proofs to Section \ref{quasitopos}.

\begin{itemize}
\item Subspaces

Any subset $Y \subseteq X$ of a smooth space $X$ becomes a smooth space 
if we define $\varphi \maps D \to Y$ to be a plot in $Y$ if and only if 
its composite with the inclusion $i \maps Y \to X$ is a plot in $X$.  
We call this the {\bf subspace} smooth structure. 

It is easy to check that with this smooth structure, the inclusion $i
\maps Y \to X$ is smooth.  Moreover, it is a monomorphism in $\cinf$.  
Not every monomorphism is of this form. 
For example, the natural map from $\R$ with its discrete smooth structure
to $\R$ with its standard smooth structure is also a monomorphism.
In Prop.\ \ref{strong.mono.prop} we show that a smooth map $i \maps Y \to X$ 
comes from the inclusion of a subspace precisely when $i$ is a `strong' 
monomorphism (see Def.\ \ref{strong.mono}).  

The 2-element set $\{0,1\}$ with its indiscrete smooth structure is 
called the `weak subobject classifier' for smooth spaces, and denoted
$\Omega$.  The precise definition of a weak subobject classifier can
be found in Def.\ \ref{subobject.class}, but the idea is simple: 
for any smooth space $X$, subspaces of $X$ are in one-to-one correspondence 
with smooth maps from $X$ to $\Omega$.  In particular, any subspace 
$Y \subseteq X$ corresponds to the characteristic function 
$\chi_Y \maps X \to \Omega$ given by
\[        \chi_Y(x) = \left\{ \begin{array}{cl} 1 & x \in Y \\
                                        0 & x \notin Y \end{array} \right.
\]
In Prop.\ \ref{subobject.prop} we prove the existence of a weak
subobject classifiers in a more general context.

\item Quotient spaces

If $X$ is a smooth space and $\sim$ is any equivalence 
relation on $X$, the quotient space $Y = X/\sim$ becomes a smooth space 
if we define a plot in $Y$ to be any function
\[ \varphi \maps D \to Y \]
for which there exists an open cover $\lbrace D_i \rbrace$ of $D$ and a collection of plots in $X$
\[ \lbrace \varphi_i \maps D_i \to X \rbrace_{i \in I}. \]
such that the following diagram commutes:

\[
\xymatrix{
 D_i \ar[r]^{\varphi_i}\ar[d]_{\iota_i} & X \ar[d]^{p} \\
 D \ar[r]_{\varphi} & Y \\
}
\]

where $p \maps X \to Y$ is the function induced by the equivalence relation $\sim$ and 
$\iota_i \maps D_i \to D$ is the inclusion.  We call this the {\bf quotient space} smooth structure.

It is easy to check that with this smooth structure, the quotient map 
$p\maps X\to Y$ is smooth, and an epimorphism in $\cinf$.  Not every
epimorphism is of this form: for example, the natural map from $\R$ with
its standard smooth structure to $\R$ with its indiscrete smooth structure
is also an epimorphism.  In Prop.\ \ref{strong.epi.prop}, we show that a smooth 
map $p \maps X \to Y$ comes from taking a quotient space precisely when $p$ 
is a `strong' epimorphism (see Def.\ \ref{strong.epi}).

\item Terminal object

The one element set $1$ can be made into a smooth space in only one way,
namely by declaring every function from every domain to $1$ to be a plot.  
This smooth space is the terminal object of $\cinf$.

\item Initial object

The empty set $\emptyset$ can be made into a smooth space in only one way,
namely by declaring every function from every domain to $\emptyset$ 
to be a plot.  (Of course, such a function exists only for the empty 
domain.)  This smooth space is the initial object of $\cinf$.

\item Products

Given smooth spaces $X$ and $Y$, the product $X \times Y$ of their underlying
sets becomes a smooth space where $\varphi \maps D \to X \times Y$ is 
a plot if and only if its composites with the projections 
\[       p_X \maps X \times Y \to X, \qquad p_Y \maps X \times Y \to Y \]
are plots in $X$ and $Y$, respectively.  We call this the {\bf product
smooth structure} on $X \times Y$.
 
It is easy to check that with this smooth structure, $p_X$ and $p_Y$ are 
smooth.  Moreover, for any other smooth space $Q$ with smooth maps 
$f_X\maps Q\to X$ and $f_Y\maps Q\to Y$, there exists a unique smooth map 
$f \maps Q \to X \times Y$ such that this diagram commutes:
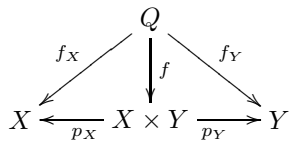
\[
\xymatrix{
  & Q\ar[d]^{f}\ar[dl]_{f_X}\ar[dr]^{f_Y} & \\
  X & X\times Y\ar[l]^{p_X}\ar[r]_{p_Y} & Y\\
}
\]
So, $X \times Y$ is indeed the product of $X$ and $Y$ in the category 
$\cinf$.

\item Coproducts

Given smooth spaces $X$ and $Y$, the disjoint union $X+Y$ of their 
underlying sets becomes a smooth space where $\varphi \maps D \to X+Y$ is a 
plot if and only if for each connected component $U$ of $D$, $\varphi|_U$ is 
either the composite of a plot in $X$ with
the inclusion $i_X \maps X \to X + Y$, or the composite of a plot in $Y$
with the inclusion $i_Y \maps Y \to X + Y$.  We call this the {\bf coproduct
smooth structure} on $X + Y$.  Note that for Chen spaces the domains of the 
plots are convex and thus have only one connected component.  So, in this
case, $\varphi$ is a plot in the disjoint union if and only if it factors 
through a plot in either $X$ or $Y$.

It is easy to check that with this smooth structure, $i_X$ and $i_Y$
are smooth.  Moreover, for any other smooth space $Q$ with smooth maps
$f_X \maps X \to Q$ and $f_Y \maps Y \to Q$, there exists a unique smooth
maps $f \maps X + Y \to Q$ such that
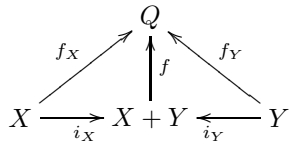
\[
\xymatrix{
  & Q & \\
  X\ar[r]_{i_X}\ar[ur]^{f_X} & X + Y\ar[u]_{f} & Y\ar[l]^{i_Y}\ar[ul]_{f_Y}\\
}
\]
commutes.  So, $X+Y$ is indeed the coproduct of $X$ and $Y$ in the category
$\cinf$.
 
\item Equalizers

Given a pair $f,g\maps X\to Y$ of smooth maps between smooth spaces, the 
set
\[ Z = \{ x\in X\colon f(x)=g(x)\} \subset X \] 
becomes a smooth space with its subspace smooth structure, and the 
inclusion $i\maps Z\to X$ is the equalizer of $f$ and $g$:
\[
\xymatrix{
Z\ar[r]^{i} & X\ar@<1ex>[r]^f \ar@<-1ex>[r]_{g} & Y
}
\]
In other words, for any smooth space $Q$ with a smooth map 
$h_X \maps Q \to X$ making this diagram commute:
\[
\xymatrix{
Q\ar[r]^{h_X} & X\ar@<1ex>[r]^f \ar@<-1ex>[r]_{g} & Y
}
\]
there exists a unique smooth
map $h \maps Q \to Z$ such that 
\[
\xymatrix{
Z\ar[r]^{i} & X\ar@<1ex>[r]^f \ar@<-1ex>[r]_{g} & Y\\
Q\ar[u]^{h}\ar[ur]_{h_X} & &
}
\]
commutes.

\item Coequalizers

Given a pair $f,g\maps X\to Y$ of smooth maps between smooth spaces, 
the quotient 
\[    Z = Y /(f(x) \sim g(x))  \]
becomes a smooth space with its quotient smooth structure, and the 
quotient map $p \maps Y \to Z$ is the coequalizer of $f$ and $g$:
\[
\xymatrix{
X\ar@<1ex>[r]^f \ar@<-1ex>[r]_{g} & Y \ar[r]^{p} & Z
}
\]
The universal property here is dual to that of the equalizer: just
turn all the arrows around.

\item Pullbacks

Since $\cinf$ has products and equalizers, it also has pullbacks,
also known as `fibered products'.  Given a diagram of smooth maps
\[
\xymatrix{
    & X\ar[d]^f \\
   Y\ar[r]_g & Z \\
}
\]
we equip the set
\[  X \times_Z Y = \{ (x,y)\in X \times Y\mid f(x) = g(y)\}  \]
with its smooth structure as a subspace of the product $X \times Y$.  
The natural functions
\[       p_X \maps X \times_Z Y \to X, \qquad p_Y \maps X \times_Z Y \to Y \]
are then smooth, and it is easy to check this diagram is a pullback square:
\[
\xymatrix{
  X \times_Z Y \ar[d]_{p_Y} \ar[r]^{p_X}  & X\ar[d]^f \\
   Y\ar[r]_g & Z \\
}
\]
In other words, given any commutative square of smooth maps
like this:
\[
\xymatrix{
  Q \ar[d]_{h_Y} \ar[r]^{h_X}  & X\ar[d]^f \\
   Y\ar[r]_g & Z \\
}
\]
there exists a unique smooth map $h \maps Q \to X \times_Z Y$ making this
diagram commute: 
\[
\xymatrix{
  & Q\ar[dr]|-{{h}}\ar@/_/[ddr]_{h_X}\ar@/^/[rrd]^{h_Y} & & \\
  &  & X \times_Z Y\ar[r]^{p_X}\ar[d]_{p_Y} & X\ar[d]^f \\
  &  & Y\ar[r]_g & Z \\
}
\]

More generally, we can compute any limit of smooth spaces by taking
the limit of the underlying sets and endowing the result with a
suitable smooth structure.  This follows from Prop.\ \ref{limit.prop},
where we show that $\cinf$ has all small limits, together with the
fact that the forgetful functor from $\cinf$ to $\Set$ preserves
limits, since it is the right adjoint of the functor equipping any set
with its discrete smooth structure.

\item Pushouts

Since $\cinf$ has coproducts and coequalizers, it also has pushouts.
Given a diagram of smooth maps
\[
\xymatrix{
   Z\ar[r]^{f}\ar[d]_{g} & X \\
   Y &  \\
}
\]
we equip the set
\[     X +_Z Y = (X + Y)/(f(z) \sim g(z))  \]
with its smooth structure as a quotient space of the coproduct $X + Y$.
The natural functions
\[      i_X \maps X \to X +_Z Y, \qquad i_Y \maps Y \to X +_Z Y  \]
are then smooth, and in fact this diagram is a pushout square:
\[
\xymatrix{
  & Z\ar[r]^{f}\ar[d]_{g} & X\ar[d]^{i_X}& \\
  & Y\ar[r]_{i_Y}& X+_Z Y \\
}
\]
The universal property here is dual to that of the pullback, and
can also be easily checked.

More generally, we can compute any limit of smooth spaces by taking
the limit of the underlying sets and endowing the result with a
suitable smooth structure.  This follows from Prop.\
\ref{colimit.prop}, where we show that $\cinf$ has all small colimits,
together with the fact that the forgetful functor from $\cinf$ to
$\Set$ preserves colimits, since it is the left adjoint of the functor
equipping any set with its indiscrete smooth structure.

\item Mapping spaces

Given smooth spaces $X$ and $Y$, the set
\[     \cinf(X,Y) =\{ f\maps X\to Y\maps f\text{ is smooth}\} \]
becomes a smooth space where a function $\tilde{\varphi} \maps D\to 
\cinf(X,Y)$ is a plot if and only if the corresponding function 
$\varphi \maps D \times X \to Y$
given by
\[             \varphi(x,y) = \tilde{\varphi}(x)(y)  \]
is smooth.  With this smooth structure one can show that the
natural map
\[      
\begin{array}{rcl}
      \cinf(X \times Y, Z) &\to& \cinf(X, \cinf(Y, Z))  \\
                   f &\mapsto& \tilde{f} 
\end{array}
\]
\[              \tilde f(x)(y) = f(x,y)  \]
is smooth, with a smooth inverse.  So, we say the category $\cinf$ is 
cartesian closed (see Def.\ \ref{cart.closed}).

\item Parametrized mapping spaces

Mapping spaces are a special case of parametrized mapping spaces.
Fix a smooth space $B$ as our parameter space, or `base space'.  
Define a smooth space {\bf over $B$} to be a smooth space $Y$ 
equipped with a smooth map $p \maps Y \to B$ called the {\bf
projection}.  For each point $b \in B$, define the 
{\bf fiber} of $Y$ over $b$ be the set
\[          Y_b = \{y \in Y \colon p(y) = b \},  \]
made into a smooth space with its subspace smooth structure.  
We can think of a smooth space over $B$ as a primitive sort of
`bundle', without any requirement of local triviality.
Note that given smooth spaces $X$ and $Y$ over $B$, the pullback
or `fibered product' $X \times_B Y$ is again a smooth space over $B$.
In fact this is the product in a certain category of smooth spaces
over $B$.

If $Y$ and $Z$ are smooth spaces over $B$, let
\[       \cinf_B(Y,Z) = \bigsqcup_{b \in B} \cinf(Y_b, Z_b). \]
We make this into a smooth space, the {\bf parametrized mapping
space}, as follows.  First define a function
\[         p \maps \cinf_B(Y,Z) \to B \]
sending each element of $\cinf(Y_b, Z_b)$ to $b \in B$. 
This will be the projection for the parametrized mapping space.
Then, note that given any smooth space $X$ and any function
\[       \tilde{f} \maps X \to C_B^\infty(Y,Z)  \]
we get a function from $X$ to $B$, namely $p \tilde{f}$.  
If this is smooth we can define the pullback smooth space
$X \times_B Y$.  Then we can define a function
\[           f \maps X \times_B Y \to Z  \]
by
\[           f(x,y) = \tilde{f}(x) (y)  .\]
This allows us to define the smooth structure on $\cinf_B(Y,Z)$: 
for any domain $D$, a function
\[          \tilde{\varphi} \maps D \to C_B^\infty(Y,Z) \]
is a plot if and only if $p \tilde{\varphi}$ is smooth and
the corresponding function 
\[  \varphi \maps D \times_B Y \to Z  \]
is smooth.  With this smooth structure, one can check that 
$p \maps \cinf_B(Y,Z) \to B$ is smooth.  So, the parametrized
mapping space is again a smooth space over $B$.

The point of the parametrized mapping space is that given smooth
spaces $X,Y,Z$ over $B$, there is a natural isomorphism of smooth 
spaces
\[   \cinf_B(X \times_B Y , Z) \cong \cinf_B(X, \cinf_B(Y,Z)) . \]
We summarize this by saying that $\cinf$ is `locally' cartesian
closed (see Def.\ \ref{loc.cart.closed}).  In the case where $B$ is a 
point, this reduces to the fact that $\cinf$ is cartesian closed.

\end{itemize}

The following theorem subsumes most of the above remarks:

\begin{definition}
A {\bf quasitopos} is a locally cartesian closed category with finite 
colimits and a weak subobject classifier.
\end{definition}

\begin{theorem} 
The category of smooth spaces, $\cinf$, is a quasitopos with all (small)
limits and colimits.
\end{theorem}

\begin{proof}
In Thm.\ \ref{quasitopos.thm} we show this holds for any 
category of `generalized spaces', that is, any category of concrete
sheaves on a concrete site.    In Prop.\ \ref{chen.prop}
we prove that $\Chen\Space$ is equivalent to a category of this 
kind, and in Prop.\ \ref{diffeological.prop} we show the same
for $\Diffeological\Space$.
\end{proof}

\section{Smooth Spaces as Generalized Spaces}
\label{generalized_spaces}

The concept of a `generalized space' was developed in the context of 
quasitopos theory by Antoine \cite{Antoine}, Penon 
\cite{Penon:1973, Penon:1977} and Dubuc \cite{Dubuc:1979,Dubuc2:2006}.  
Generalized spaces form a natural framework for studying Chen spaces, 
diffeological spaces, and even simplicial complexes.  For us, a category 
of generalized spaces will be a category of `concrete sheaves' over a 
`concrete site'.  For a self-contained treatment, we start by
explaining some basic notions concerning sheaves and sites.  We motivate
all these notions with the example of Chen spaces, and 
in Prop.\ \ref{chen.prop}, we prove that Chen spaces are
concrete sheaves on a concrete site.  We also prove
similar results for diffeological spaces and simplicial complexes.

We can define sheaves on a category as soon as we have a 
good notion of when a family of morphisms $f \maps D_i \to D$ `covers'
an object $D$.  For this, our category should be what is called a `site'.  
Usually a site is defined to be a category equipped with a 
`Grothendieck topology'.  However, as emphasized by Johnstone 
\cite{Johnstone2}, we can get away with less: it is enough to use 
a `Grothendieck pretopology', or `coverage'.  The difference is not very 
great, since every coverage on a category determines a
Grothendieck topology with the same sheaves.  Coverages 
are simpler to define, and for our limited purposes they are easier to 
work with.   So, we shall take a site to be a category equipped with a 
coverage.  Two different coverages may determine the same
Grothendieck topology, but knowledgeable readers can check that
everything we do depends only on the Grothendieck topology.

\begin{definition}
A {\bf family} is a collection of morphisms with common codomain.
\end{definition}

\begin{definition}
A {\bf coverage} on a category $\D$ is a function 
assigning to each object $D\in\D$ a collection $\mathcal{J}(D)$ of 
families $(f_i\maps D_i\to D|i\in I)$ called {\bf covering
families}, with the following property:
\begin{itemize}
\item 
Given a covering family $(f_i\maps D_i\to D|i\in I)$ and a
morphism $g\maps C \to D$, there 
exists a covering family $(h_j\maps C_j\to C|j\in J)$ such 
that each morphism $gh_j$ factors through some $f_i$.
\end{itemize}
\end{definition}

\begin{definition}
A {\bf site} is a category equipped with a coverage.  We
call the objects of a site {\bf domains}.
\end{definition}

In Lemma \ref{chen.lem} we describe a coverage
on the category $\Chen$, whose objects are convex sets
and whose morphisms are smooth functions.  For this coverage, 
a covering family is just an open
cover in the usual sense.  This makes $\Chen$ into a site, and 
Chen spaces will be `concrete sheaves' on this site.  
To understand how this works, let us quickly review sheaves and 
then explain the concept of `concreteness'.

\begin{definition}
A {\bf presheaf} $X$ on a category $\D$ is a functor
$X \maps \D^\op \to \Set$.  For any object $D \in \D$, we call 
the elements of $X(D)$ {\bf plots in $X$ with domain $D$}.
\end{definition}

\noindent
Usually the elements of $X(D)$ are called `sections of $X$ over
$D$'.  However, given a Chen space $X$ there is a presheaf on 
$\Chen$ assigning to any convex set $D$ the set $X(D)$ of all 
plots $\varphi \maps D \to X$.  So, it will guide our intuition 
to quite generally call an object $D \in \D$ a `domain' and 
elements of $X(D)$ `plots'.

Axiom 1 in the definition of a Chen space is what gives us a
contravariant functor from $\Chen$ to $\Set$: it says that given 
any morphism $f \maps C \to D$ in $\Chen$, we get a function
\[          X(f) \maps X(D) \to X(C)  \]
sending any plot $\varphi \maps D \to X$ to the plot $\varphi f \maps
C \to X$.  Axiom 2 says that the resulting presheaf on $\Chen$
is actually a sheaf:

\begin{definition}
Given a covering family $(f_i\maps D_i\to D|i\in I)$ in $\D$ and a presheaf 
$X\maps\D^\op\to\Set$, a collection of plots 
$\{\varphi_i\in X(D_i)|i\in I\}$ is called {\bf compatible} if whenever 
$g\maps C\to D_i$ and $h\maps C\to D_j$ make this diagram commute:
\[
\xymatrix{
C\ar[r]^{h}\ar[d]_{g} & D_j\ar[d]^{f_j}\\
D_i\ar[r]_{f_i} & D
}
\]
then $X(g)(\varphi_i)=X(h)(\varphi_j)$.
\end{definition}

\begin{definition} \label{sheaf}
Given a site $\D$, a presheaf $X\maps\D^\op\to\Set$ 
is a {\bf sheaf} if it satisfies the following condition:
\begin{itemize}
\item Given a covering family $(f_i\maps D_i\to D|i\in I)$ 
and a compatible collection of plots $\{\varphi_i\in X(D_i)|i\in I\}$, 
then there exists a unique plot $\varphi\in X(D)$ such that 
$X(f_i)(\varphi) = \varphi_i$ for each $i\in I$.
\end{itemize}
\end{definition}

On any category, there is a special class of presheaves called the
`representable' ones:

\begin{definition}
A presheaf $X\maps\D^\op\to\Set$ is called {\bf representable} if it
is naturally isomorphic to $\hom(-,D)\maps\D^\op\to\Set$ for some
$D\in\D$.  
\end{definition}

\noindent
The site $\Chen$ is `subcanonical':

\begin{definition}
A site is {\bf subcanonical} if every representable presheaf on
this site is a sheaf.  
\end{definition}

\noindent
We shall include this property in the definition of `concrete site'.
But there is a much more important property that we shall also
require.  A Chen space $X$ gives a special kind of sheaf on the site
$\Chen$: a `concrete' sheaf, meaning roughly that for any $D \in
\Chen$, elements of $X(D)$ are certain functions from the underlying
set of $D$ to some fixed set.  Of course, this notion relies on the
fact that $D$ has an underlying set!  The following definition ensures
that this is the case for any object $D$ in a concrete site:

\begin{definition}\label{concrete_site}
A {\bf concrete site} $\D$ is a subcanonical site with a terminal 
object $1$ satisfying the following conditions:
\begin{itemize}
\item The functor $\hom(1,-)\maps\D\to\Set$ is faithful.
\item For each covering family $(f_i\maps D_i\to D|i\in I)$, the family 
of functions $(\hom(1,f_i)\maps\hom(1,D_i)\to\hom(1,D)|i\in I)$ is 
{\bf jointly surjective}, meaning that the union of their images
is all of $\hom(1,D)$.
\end{itemize}
\end{definition}

\noindent
Quite generally, any object $D$ in a category $\D$ with a 
terminal object has an underlying set $\hom(1,D)$, often
called its set of `points'.  The requirement 
that $\hom(1,-)$ be faithful says that two morphisms 
$f,g \maps C \to D$ in $\D$ are equal when they induce
the same functions from points of $C$ to points of $D$.  
In other words: objects have enough points to distinguish 
morphisms.  In this situation we can think of objects of 
$\D$ as sets equipped with extra structure.  The second 
condition above then says that the underlying
family of functions of a covering family is itself a
`covering', in the sense of being jointly surjective.

Henceforth, we let $\D$ stand for a concrete site.  
Now we turn to the notion of `concrete sheaf'.  
There is a way to extract a set from a sheaf on a concrete site.  
Namely, a sheaf $X \maps \D^\op \to \Set$   
gives a set $X(1)$.  In the case of a sheaf coming from a Chen 
space, this is the set of one-point plots $\varphi \maps 1 \to X$.  
Axiom 3 implies that it is the underlying set of the Chen space.  
Furthermore, for any sheaf $X$ on a concrete site, there is a way 
to turn a plot $\varphi \in X(D)$ into a function $\u{\varphi}$ 
from $\hom(1,D)$ to $X(1)$.   To do this, set
\[          \u{\varphi}(d) = X(d)(\varphi)  .  \]
A simple computation shows that for the sheaf coming from a Chen 
space, this process turns any plot into its underlying function.  
(See Prop.\ \ref{chen.prop} for details.)  In this example, we 
lose no information when passing from $\varphi$ to the
function $\u{\varphi}$: distinct plots have distinct underlying 
functions.  The notion of `concrete sheaf' makes this idea 
precise quite generally:

\begin{definition} \label{concrete_sheaf}
Given a concrete site $\D$, we say a sheaf $X \maps \D^\op \to \Set$
is {\bf concrete} if for every object $D \in \D$, the
function sending plots $\varphi \in X(D)$ to functions
$\u{\varphi} \maps \hom(1,D) \to X(1)$ is one-to-one.
\end{definition}

We can think of concrete sheaves as `generalized spaces', since 
they generalize Chen spaces and diffeological spaces. 
Every concrete site gives a category of generalized spaces:

\begin{definition} \label{dspace}
Given a concrete site $\D$, a {\bf generalized space} or {\bf $\D$ space} 
is a concrete sheaf $X\maps\D^\op \to\Set$.  A {\bf map} between 
$\D$ spaces $X,Y \maps \D^\op \to \Set$ is a natural transformation 
$F\maps X \To Y$.   We define $\mathbf{\D\mathrm{Space}}$ to be the 
category of $\D$ spaces and maps between these.
\end{definition}

Now let us give some examples:

\begin{lemma}\label{chen.lem} 
Let $\Chen$ be the category whose objects are convex sets and
whose morphisms are smooth functions.  The category
$\Chen$ has a subcanonical coverage where $(i_j\maps C_j\to C|j\in J)$ 
is a covering family if and only if the convex sets $C_j \subseteq
C$ form an open covering of the convex
set $C \subseteq \R^n$ with its usual subspace topology, and 
$i_j \maps C_j \to C$ are the inclusions.  
\end{lemma}

\begin{proof}
Given such a covering family $(i_j\maps C_j\to C|j \in J)$ 
and $g\maps D\to C$ in $\Chen$, then 
$\{ g^{-1}(i_j(C_j))\}$ is an open cover of $D$ which 
factors through the family $i_j$ as functions on sets.  We can refine 
this cover by convex open balls to obtain a covering family of $D$ which 
factors through the family $i_j$ in $\Chen$.  Since the covers are open
covers in the usual sense, it is clear that the site is subcanonical.
\end{proof}

We henceforth consider $\Chen$ as a site with the above coverage.  
Since any 1-point convex set is a terminal object, $\Chen$ is a 
concrete site.  This allows us to define a kind of generalized space 
called a `$\Chen$ space' following Def.\ \ref{dspace}.  

\begin{proposition} \label{chen.prop}
A Chen space is the same as a $\Chen$ space.  More precisely, 
the category of Chen spaces and smooth maps is equivalent to the 
category $\Chen\Space$.
\end{proposition}

\begin{proof} Let $\cinf$ stand for the category of Chen spaces
and smooth maps. We begin by constructing functors from $\cinf$ to 
$\Chen\Space$ and back.  To reduce confusion, just for now we use 
italics for objects and morphisms in $\cinf$, and boldface 
for those in $\Chen\Space$.

First, given $X \in \cinf$, we construct a concrete sheaf $\X$ on 
$\Chen$.  For each convex set $C$, we define $\X(C)$ to be 
the set of all plots $\varphi\maps C\to X$, and given a smooth function 
$f \maps C' \to C$ between convex sets, we define $\X(f) \maps
\X(C) \to \X(C')$ as follows:
\[        \X(f) \varphi = \varphi f .\]
Axiom 1 in Chen's definition guarantees that $\varphi f$ lies
in $\X(C')$, and it is easy to check that $\X$ is a presheaf.
Axiom 2 ensures that this presheaf is a sheaf.  

To check that $\X$ is concrete, note first that
axiom 3 gives a bijection between underlying set of $X$ and the set 
$\X(1)$, sending any point $x \in X$ to the one-point plot 
whose image is $x$.  Then, let $\varphi \in \X(C)$ and compute 
$\u{\varphi} \maps \hom(1,C) \to X(1) \cong X$:
\[     \u{\varphi}(c) = X(c)(\varphi) = \varphi(c)  \]
where at the last step we identify the smooth function $c \in \hom(1,C)$ 
with the one point in its image.  So, $\u{\varphi}$ is the underlying
function of the plot $\varphi$.  It follows that the map sending $\varphi$
to $\u{\varphi}$ is one-to-one, so $\X$ is concrete. 

Next, given a smooth map $f \maps X \to Y$ between Chen spaces,
we construct a natural transformation $\f \maps X \to \Y$ between the 
corresponding sheaves.  For this, we define
\[                \f_C \maps \X(C) \to Y(C) \]
by
\[                \f_C (\varphi) = f \varphi \]
To show that $\f$ is natural, we need the following square to commute
for any smooth function $g \maps C' \to C$:
\[
\xymatrix{
\X(C)\ar[r]^{\f_C}\ar[d]_{\X(g)} & \Y(C)\ar[d]^{\Y(g)}\\
\X(C')\ar[r]_{\f_{C'}} & \Y(C')
}
\]
This just says that $(f \varphi) g = f (\varphi g)$.

We leave it to the reader to verify that this construction defines a functor
from $\cinf$ to $\Chen\Space$.

To construct a functor in the other direction, we must first construct
a Chen space $X$ from any concrete sheaf $\X$ on $\Chen$.  For this we take 
$X = \X(1)$ as the underlying set of the Chen space, and take as plots 
in $X$ with domain $C$ all functions of the form $\u{\varphi}$ where 
$\varphi \in \X(C)$.  
Axiom 1 in the definition of Chen space follows from the fact that $\X$ is a 
presheaf.  Axiom 2 follows from the fact that $\X$ is a sheaf.  Axiom 3 
follows from the fact that $X = \X(1)$.  Next, we must construct 
a function $f \maps X \to Y$ from a natural transformation $\f \maps \X \to \Y$
between concrete sheaves.  For this we set
\[       f = \f_1 \maps \X(1) \to \Y(1)  .\]
Again, we leave it to the reader to check that this construction defines
a functor.

Finally, we must check that the composite of these functors in either order is 
naturally isomorphic to the identity.   This is straightforward in the
case where we turn a Chen space $X \in \cinf$ into a concrete sheaf $\X$ and 
back into a Chen space.  When we turn a concrete sheaf $\X$ into a Chen 
space $X$ and back into a concrete sheaf $\X'$, we have
\[        \X'(C) = \{ \u{\varphi} \maps C \to X(1) \}  \]
but the latter is naturally isomorphic to $X(C)$ via the function
\[      \begin{array}{rcl}
      X(C) &\to& X'(C)  \\
      \varphi & \mapsto & \u{\varphi} 
\end{array}
\]
thanks to the fact that $\X$ is concrete.
\end{proof}

Diffeological spaces work similarly:

\begin{lemma}\label{diffeological.lem} 
Let $\Diffeological$ be the category whose objects are open subsets of 
$\R^n$ and whose morphisms are smooth maps.  The category $\Diffeological$ 
has a subcanonical coverage where $(i_j\maps U_j\to U|j\in J)$ 
is a covering family if and only if the open sets $U_j \subseteq
U$ form an open covering of the open set $U \subseteq \R^n$, 
and $i_j \maps U_j \to U$ are the inclusions.  
\end{lemma}

\begin{proof}
The proof is a simpler version of the proof for $\Chen$, since we are 
considering open but not necessarily convex sets.
\end{proof}

We henceforth treat $\Diffeological$ as a site with this coverage.  
The one-point open subset of $\R^0$ is a terminal object
for $\Diffeological$, so this is a concrete site.  As before, we have:

\begin{proposition} \label{diffeological.prop}
A diffeological space is the same as a $\Diffeological$ space.  
More precisely, the category of diffeological spaces is equivalent 
to the category \hfill \break $\Diffeological\Space$.
\end{proposition}

\begin{proof}
The proof of the corresponding statement for Chen spaces applies here as 
well.
\end{proof}

An example of a very different flavor is the category of simplicial
complexes:

\begin{definition}
An {\bf (abstract) simplicial complex} is a set $X$ together with
a family $K$ of nonempty finite subsets of $X$ such that:
\begin{enumerate}
\item Every singleton lies in $K$.
\item If $S \in K$ and $T \subseteq S$ then $T\in K$. 
\end{enumerate}
A {\bf map} of simplicial complexes $f\maps (X,K)\to (Y,L)$  is a function 
$f\maps X \to Y$ such that $S\in K$ implies $f(S)\in L$.
\end{definition}
\noindent
We can geometrically realize any simplicial complex $(X,K)$ by turning 
each $n$-element set $S \in K$ into a geometrical $(n-1)$-simplex.  Then
axiom 1 above says that any point of $X$ corresponds to a 0-simplex, while
axiom 2 says that any face of a simplex is again a simplex.  

To view the category of simplicial complexes as a category
of generalized spaces, we use the following site:

\begin{lemma} \label{simplicial.lem}
Let $\Simplicial$ be the category with nonempty finite sets as objects
and functions as morphisms.  There is a subcanonical coverage on $\Simplicial$
where for each object $D$ in $\Simplicial$ there is exactly one
covering family, consisting of all inclusions $D' \hookrightarrow D$.
\end{lemma}

\begin{proof}
Given a covering family $(f_i\maps D_i\hookrightarrow D|i\in I)$ and a
function $g\maps C \to D$, each function in a covering family having
$C$ as codomain composed with $g$ clearly factors through some $f_i$.
For instance, take $f_i$ to be the identity function on $D$.  The coverage
is clearly subcanonical since each covering includes the identity morphism.
\end{proof}

Henceforth we make $\Simplicial$ into a concrete site with the 
above coverage.  Since every covering family contains the identity,
this coverage is `vacuous': every presheaf is a sheaf.   Presheaves 
on $\Simplicial$ have been studied by Grandis under the name {\bf symmetric 
simplicial sets}, since they resemble simplicial sets whose
simplices have unordered vertices \cite{Grandis}.  It turns out that 
{\it concrete} sheaves on $\Simplicial$ are simplicial complexes:

\begin{proposition} \label{simplicial.prop}
The category of $\Simplicial$ spaces 
is equivalent to the category of simplical complexes.
\end{proposition}

\begin{proof}
We define a functor from the category of $\Simplicial$ spaces to the
category of simplicial complexes.   We use $n$ to stand for an 
$n$-element set.  Since the underlying set $\hom(1,n)$ of 
$n \in \Simplicial$ is naturally isomorphic to $n$, we shall not bother 
to distinguish between the two.

Given an $\Simplicial$ space, that is a concrete sheaf $\X \maps 
\Simplicial^{op}\to\Set$, we define a simplicial complex $(X,S)$
with $X = \X(1)$ and $K= \{\im \u{\varphi}| \varphi\in \X(n), n \in 
\Simplicial \}$.
To check axiom 1, we note that a point $x \in X$ is a plot $\varphi \in \X(1)$, and $\lbrace x \rbrace = \im \u{\varphi} \in K$.  To check axiom 2, we fix an object $n$, a plot
$\varphi\in \X(n)$ and a subset $Y\subseteq \im \u{\varphi} \in K$.
We consider $\u{\varphi}^{-1}(Y)\subseteq n$ and let $m$ be the object
in $\Simplicial$ representing the finite set of cardinality
$|\u{\varphi}^{-1}(Y)|$.  There is an inclusion $m \hookrightarrow
\u{\varphi}^{-1}(Y)\subseteq n$ and commutativity of
\[
\xymatrix{
S(n)\ar[r]\ar@{ >->}[d] & S(m)\ar@{ >->}[d]\\
\hom(n,S(1))\ar[r] & \hom(m,S(1))
}
\]
shows that $Y$ is an element of $K$ and that the structure defined is,
in fact, a simplicial complex.

Given a natural transformation 
$\f \maps \X\Rightarrow \Y$ between $\Simplicial$ spaces we obtain a map 
$f = \f_1 \maps \X(1) \to \Y(1)$.  By the commutativity of
\[
\xymatrix{
\X(n)\ar[r]\ar@{ >->}[d] & \X(n)\ar@{ >->}[d]\\
\hom(n,\X(1))\ar[r] & \hom(n,\Y(1))
}
\]
we see that this defines a map of simplicial complexes and this process 
clearly preserves identities and composition.  Since a map $\f \maps 
\X \Rightarrow \Y$ of $\Simplicial$ spaces is completely determined by the 
function $\f_1\maps \X(1)\to \Y(1)$ it is clear that this functor is
faithful.  We see that the functor is full since given a map of
simplicial complexes $f\maps (X,K)\to (Y,L)$ and a morphism between finite
sets $j\maps m\to n$, then the naturality square
\[
\xymatrix{
\X(n)\ar[d]_{\f(n)}\ar[r]^{\X(j)} & \X(m)\ar[d]^{\f(m)}\\
\Y(n)\ar[r]_{\Y(j)} & \Y(m)
}
\]
commutes, thus defining a natural transformation between $\Simplicial$
spaces.

We can also reverse the process described, taking a
simplicial complex $(X,K)$ and defining an $\Simplicial$ space $\X$
whose image is isomorphic to $(X,K)$.  For each $n \in \Simplicial$,
we let $\X(n)$ be the set of $n$-element sets $S \in K$.  The
downward closure property of simplicial complexes guarantees that this
is a $\Simplicial$ space and it is easy to see that one can construct
an isomorphism from the image of this $\Simplicial$ space under our
functor to $(X,K)$.  Thus, we have obtained an equivalence of categories.
\end{proof}

\section{Convenient Properties of Generalized Spaces}
\label{quasitopos}

In this section we establish convenient properties of any category 
of generalized spaces.   We begin with some handy notation.  In Section 
\ref{generalized_spaces} we introduced three closely linked notions of 
`underlying set' or `underlying function' in the context of a concrete site 
$\D$.  It will now be convenient, and we hope not confusing, to denote all 
three of these by an underline:

\begin{itemize}
\item The underlying set of a domain: $\u{D} = \hom(1,D)$

Any concrete site $\D$ has an `underlying set' functor 
$\hom(1,-) \maps \D \to \Set$.  Henceforth we denote this 
functor by an underline:
\[            \u{\;\;} \maps \D \to \Set  . \]
So, any domain $D \in \D$ has an underlying set $\u{D}$, and any morphism 
$f \maps C \to D$ in $\D$ has an underlying function $\u{f} \maps \u{C} \to 
\u{D}$.   The concreteness condition on $\D$ says that this underlying set 
functor is faithful.

\item The underlying set of a generalized space: $\u{X} = X(1)$

Any generalized space $X\maps \D^{\op} \to \Set$ has an 
underlying set $X(1)$.   Henceforth we denote this set as $\u{X}$.  Similarly,
any map of generalized spaces $f \maps X \to Y$ has an 
underlying function $f_1 \maps X(1) \to Y(1)$, which we henceforth write 
as $\u{f}\maps \u{X} \to \u{Y}$.   It is easy to check that these combine
to give an `underlying set' functor
\[        \u{\;\;} \maps \D\Space \to \Set . \]
In Proposition \ref{faithful.prop} we show that this underlying set
functor is also faithful. 

\item The underlying function of a plot: $\u{\varphi}(d) = X(d)(\varphi)$

For any generalized space $X\maps \D^{\op} \to \Set$, any plot 
$\varphi\in X(D)$ has an underlying function $\u{\varphi}\maps \u{D} \to 
\u{X}$ defined as above.   The concreteness condition in the definition of  
`generalized space' says the map from plots to their underlying functions is 
one-to-one.  One can check that this map defines a natural transformation
\[
\u{\;\;} \maps X(D) \to \u{X}^{\u{D}}\, .\]

\end{itemize}

\begin{proposition} \label{faithful.prop} 
The underlying set functor $\u{\;\;} \maps \D\Space \to \Set$ is faithful.   
\end{proposition}

\begin{proof}
Given $\D$ spaces $X$ and $Y$, suppose $f, g \maps X \to Y$ have
$\u{f}=\u{g}$.  We need to show that $f = g$.  Recall that
$f$ and $g$ are natural transformations between the functors
$X,Y \maps \D^\op \to \Set$, so given $D\in\D$ the 
following squares commute for each $d\in \u{D}$:
\[
\xymatrix{
X(D)\ar[r]^{f_D}\ar[d]_{X(d)} & Y(D)\ar[d]^{Y(d)} 
& & X(D)\ar[r]^{g_D}\ar[d]_{X(d)} & Y(D)\ar[d]^{Y(d)} \\
X(1)=\u{X}\ar[r]_{\u{f}} & \u{Y}=Y(1) 
& & X(1)=\u{X}\ar[r]_{\u{g}} & \u{Y}=Y(1)
}
\]
We need to show that for any
$\varphi\in X(D)$, $f_D(\varphi) = g_D(\varphi)$ in $Y(D)$. 
Since the natural transformation
\[       \u{\;\;} \maps Y(D) \to \u{Y}^{\u{D}}  \]
is one-to-one, it suffices to show that 
\[   \u{f_D(\varphi)}(d) = \u{g_D(\varphi)}(d) \] 
for all $d\in D$, or in other words,
\[       Y(d) f_D(\varphi) = Y(d) g_D(\varphi)  .\]
By the above commuting squares, this amounts to showing
\[         \u{f} X(d)(\varphi) = \u{g} X(d)(\varphi)  \]
but this follows from $\u{f} = \u{g}$.
\end{proof}

There is a further relation between the two `underlying set' functors
mentioned above.  In the case of Chen spaces, every convex set 
naturally becomes a Chen space with the same underlying set.  This 
happens quite generally:

\begin{proposition}\label{representable.prop}
Every representable presheaf on $\D$ is a $\D$ space.  The underlying
set of the $\D$ space $\hom(-,D) \maps \D^\op \to \Set$ is equal to
$\u{D}$.
\end{proposition}

\begin{proof}
Since a concrete site is subcanonical by definition, every
representable presheaf on $\D$ is a sheaf.  So, to show that the
representable presheaves are $D$ spaces, we just need to show that
they are \textsl{concrete} sheaves.  Suppose $X \cong \hom(-, D)$ is a
representable presheaf.  Then given $C\in\D$, the map $\u{\;\;} \maps
\hom(C,D) \to {\u{C}}^{\u{D}}$ takes a morphism $f\maps C\to D$ to its
underlying function $\u{f}\maps \u{C} \to \u{D}$ and thus is
one-to-one.  It follows that $\hom(-,D)$ is concrete.  The underlying
set of this $\D$ space is $\hom(1,D)$, which is just $\u{D}$.
\end{proof}

\subsection{Subspaces, Quotient Spaces, and Limits}

With these preliminaries in hand, we now study subspaces and
quotient spaces of $\D$ spaces, and show that the category of $\D$
spaces has a weak subobject classifier, $\Omega$.  In the process
we will show that $\D\Space$ has limits.

For Chen spaces or 
diffeological spaces, $\Omega$ is just the 2-element set $2 = \{0,1\}$ 
equipped with its indiscrete smooth structure.  In general, $\Omega$ 
will have the 2-element set as its underlying set, and for any $D \in \D$,
every function $\varphi \maps \u{D} \to 2$ will count as a plot.
So, $\Omega(D)$ will be the power set of $\u{D}$:

\begin{proposition} \label{omega}
There is a $\D$ space $\Omega$ such that for any object $D \in \D$, 
$\Omega(D) = 2^{\u{D}}$, and for any morphism $f \maps C \to D$ in $\D$,
$\Omega(f) \maps 2^{\u{D}} \to 2^{\u{C}}$ sends any plot $\varphi \maps
\u{D} \to 2$ to the plot $\varphi \u{f} \maps \u{C} \to 2$.
\end{proposition}

\begin{proof}
$\Omega$ is clearly a presheaf.   To show that it is a sheaf, we 
suppose $(f_i \maps D_i \to D|i \in I)$ is a covering family 
and $\{\varphi_i \in \Omega(D_i)|i \in I\}$ is a compatible family of 
plots, and show there exists a unique plot $\varphi \in \Omega(D)$ with 
$\Omega(f_i) \varphi = \varphi_i$.  The compatible family of plots
consists of functions $\varphi_i \maps \u{D_i} \to 2$, and we need
to show there exists a unique function $\varphi \maps \u{D} \to 2$
with $\varphi f_i = \varphi_i$.

For the existence of $\varphi$, suppose $d \in \u{D}$.  Then since the
family $(\u{f}_i\maps \u{D}_i\to \u{D})$ is jointly surjective we can
find an $i$ such that there exists $d'\in\u{D_i}$ with $\u{f_i}(d')=d$.
We define $\varphi(d)=\varphi_i(d')$.  To show that the $\varphi(d)$ is
independent of the choice of $i$, suppose that
$d'\in\u{D_i}\cap\u{D_j}$ and consider morphisms $g\maps 1\to D_i$ and
$h\maps 1\to D_j$ such that $\u{g}(1)=d'=\u{h}(1)$.  Since the plots
were chosen to be compatible with the family, we have
$\Omega(g)(\varphi_i)=\Omega(h)(\varphi_j)$.  In other words,
$\varphi_i(d')=\varphi_j(d')$.  Uniqueness follows from the family
being jointly surjective.  Finally, since plots $\varphi \in
\Omega(D)$ are in one-to-one correspondence with functions
$\u{\varphi} \maps \u{D} \to 2$, the sheaf $\Omega$ is concrete.
\end{proof}

\begin{proposition} \label{epi.mono}
A monomorphism (resp.\ epimorphism) in $\D\Space$ is a map 
$f\maps X\to Y$ for which the underlying function $\u{f}$ is injective 
(resp.\ surjective). 
\end{proposition}

\begin{proof}
For one direction, recall that $\u{\;\;} \maps \D\Space \to \Set$ is 
faithful by Prop.\ \ref{faithful.prop}, so a morphism 
$f \maps X \to Y$ in $\D\Space$ is monic (resp.\ epic) 
if its image under this functor has the same property.   

Conversely, suppose $f$ is monic.  Then the map from $\hom(1,X)$ to 
$\hom(1,Y)$ given by composing with $f$ is injective, but this 
says precisely that $\u{f}$ is injective.  

Next, suppose $f$ is epic.  Then the map from $\hom(Y,\Omega)$ to
$\hom(X,\Omega)$ given by composing with $f$ is injective, but
this says that the map from $2^{\u{Y}}$ to $2^{\u{X}}$ sending
$\chi \maps \u{Y} \to 2$ to $\chi f \maps \u{X} \to 2$ is injective, 
which implies $f$ is surjective.  \end{proof}

\begin{definition} \label{strong.mono}
In any category, a monomorphism $i\maps A\to X$ is {\bf strong} if 
given any epimorphism $p\maps E\to B$ and morphisms $f,g$ making the
outer square here commute:
\[
\xymatrix{
E \ar[r]^{f}\ar@{-> >}[d]_{p} & A \ar@{ >->}[d]^{i}\\
B \ar[r]_{g}\ar@{-->}[ur]|-{{t}}   & X\\
}
\]
then there exists a unique $t\maps B \to A$ 
making the whole diagram commute.
\end{definition}

\begin{definition} \label{subspace}
We say a morphism of $\D$ spaces $i \maps A \to X$ makes
$A$ a {\bf subspace} of $X$ if for any plot $\varphi\in X(D)$ 
with $\u{\varphi}(\u{D}) \subseteq \u{i}(\u{A})$, there 
exists a unique plot $\psi \in A(D)$ with $i_D(\psi) = \varphi$.
\end{definition}

\begin{proposition} \label{strong.mono.prop} A morphism of $\D$ spaces 
$i\maps A \to X$ is a strong monomorphism if and only if $i$ makes $A$ 
a subspace of $X$.  
\end{proposition} 
\begin{proof} Suppose $i \maps A \to X$ is a subspace of $X$.  
Given an epimorphism $p\maps E\to B$ and morphisms $f,g$ such that 
the outer square here:
\[
\xymatrix{
E \ar[r]^{f}\ar@{-> >}[d]_{p} & A \ar@{ >->}[d]^{i}\\
B \ar[r]_{g}\ar@{-->}[ur]|-{{t}}   & X\\
}
\]
commutes, we need to prove there exists a unique $t \maps E \to B$
making the whole diagram commute.  Define functions $t_D\maps B(D) \to A(D)$
as follows.  Note that for any plot $\varphi \in B(D)$, the plot
$g_D(\varphi) \in X(D)$ has    
\[ \u{g_D (\varphi)}(\u{D}) = \u{g} \u{\varphi} (\u{D}) \subseteq \u{i}(\u{A}) , \]
where in the first step we use the naturality of the map sending a plot to
its underlying function, and in the second we use the commutative diagram
of underlying functions.  By Def.\ \ref{subspace} it follows that there exists 
a unique plot $\psi \in A(D)$ with $i_D(\psi) = g_D(\varphi)$.  We set 
\[        t_D(\varphi) = \psi  .\]
We can check that $t$ is a natural transformation by considering a morphism 
$f\maps D'\to D$ in $\D$ and the following diagram:

\[
\xymatrix{
				&&&
			\u{A}\ar@{>->}[d]^{\u{i}}\\
\u{D'}\ar@{-->}[urrr]\ar[r]_{\u{f}} & \u{D}\ar[r]_{\u{\varphi}}\ar@{-->}[urr]|-{{\u{\psi}}} & \u{B}\ar[r]_{\u{g}} & \u{X}\\
}
\]
By the description of $t_D$ above, we see that $\u{g}\u{\varphi}$ is
uniquely lifted to a plot $\u{\psi}\maps\u{D}\to\u{A}$.  Then
$\u{g}\u{\varphi}\u{f}$ also has a unique lift, which must be
$\u{\psi}\u{f}\maps\u{D'}\to\u{A}$.  We have seen that the naturality
square
\[
\xymatrix{
B(D)\ar[r]^{t_D}\ar[d]_{B(f)} & A(D)\ar[d]^{A(f)}\\
B(D')\ar[r]_{t_{D'}} & A(D')
}
\]
commutes, and thus that $t$ is a map of $\D$ spaces.  The lower triangle
commutes by construction.  The upper triangle commutes since
$f=i^{-1}gp=tp$.  We can check that $t$ is unique at the level of the
underlying functions, where it follows from the commutativity of the
diagram and that $i$ is a monomorphism.  Now we have shown that $i$ is
a strong monomorphism.

Conversely, suppose $i\maps A\to X$ is a strong monomorphism and
consider a plot $\varphi\in X(D)$ for some $D\in\D$ with
$\u{\varphi}(\u{D})\subseteq\u{i}(\u{A})$.  We give the set
$\u{A}':=\u{\varphi}(\u{D})\subseteq\u{X}$ the subspace structure from
$X$ and we give $\u{A}'':=\u{i}^{-1}\u{\varphi}(\u{D})$ the subspace
structure of $A$.  Then a $\D\Space$ epimorphism from $A''$ to $A'$ is
induced by restricting $\u{i}$ and we have the following commutative
diagram:
\[
\xymatrix{
A''\ar@{ >->}[r]\ar@{-> >}[d] & A\ar@{ >->}[d]^{i}\\
A'\ar@{ >->}[r]_{j}\ar@{-->}[ur]|-{{t}} & X
}
\]
where $t$ exists and is unique since $i$ is a strong monomorphism.  Since
$A'$ is a subspace of $X$ and $\u{\varphi}(\u{D})=\u{A}'$, there
exists a unique plot $\psi\in A'(D)$ such that $j_D(\psi)=\varphi$.
Thus we have $t_D(\psi)\in A(D)$ and by commutativity of the diagram
$i_Dt_D(\psi)=j_D(\psi)=\varphi$.  For any other $\psi'\in A(D)$ with
$i_D(\psi')=\varphi$, we have $\psi'=t_D(\psi)$ since $i$ is a
monomorphism, and thus $t_D(\psi)$ is unique as desired.
\end{proof}

\begin{definition} \label{strong.epi}
In any category, an epimorphism $p\maps E\to B$ is {\bf strong} if 
given any monomorphism $i\maps A\to X$ and morphisms $f,g$ making the
outer square here commute:
\[
\xymatrix{
E \ar[r]^{f}\ar@{-> >}[d]_{p} & A \ar@{ >->}[d]^{i}\\
B \ar[r]_{g}\ar@{-->}[ur]|-{{t}}   & X
}
\]
then there exists a unique $t\maps B \to A$ 
making the whole diagram commute.
\end{definition}

\begin{definition}
We say a morphism of $\D$ spaces $p \maps E \to B$ makes
$B$ a {\bf quotient space} of $E$ if for every plot $\varphi \in B(D)$,
there exists a covering family $(f_i \maps D_i \to D | i \in I)$ in $\D$ and a
collection of plots $\lbrace \varphi_i \in E(D_i) | i \in I \rbrace$ such that
the following diagram commutes:

\[
\xymatrix{
 D_i \ar[r]^{\varphi_i}\ar[d]_{f_i} & E \ar[d]^{p} \\
 D \ar[r]_{\varphi} & B \\
}
\]
\end{definition}

With this definition, the underlying map $\u{p}\maps \u{E} \to \u{B}$
is a surjection and thus defines an equivalence relation, $e_1\sim
e_2$ if and only if $\u{p}(e_1)=\u{p}(e_2)$, such that $\u{B}$ is the
quotient of $\u{E}$ by this equivalence relation.  The extra condition
that every plot in $B$ comes locally from a plot in $E$ gives the
following theorem:

\begin{proposition} \label{strong.epi.prop}
A morphism of $\D$ spaces $p\maps E \to B$ is a strong epimorphism
if and only if $p$ makes $B$ a quotient space of $E$.
\end{proposition}

\begin{proof}
Suppose that $p\maps E \to B$ makes $B$ a quotient space of $E$.  By Prop.\ 
\ref{epi.mono}, $p$ is an epimorphism since $\u{p}$ is surjective.  
To show $p$ is a strong epimorphism, given any monomorphism $i\maps A\to X$ 
and morphisms $f,g$ making the outer square commute:
\[
\xymatrix{
E \ar[r]^{f}\ar@{-> >}[d]_{p} & A \ar@{ >->}[d]^{i}\\
B \ar[r]_{g}\ar@{-->}[ur]|-{{t}}   & X\\
}
\]
we need to prove there exists a unique $t\maps B \to A$ making the
whole diagram commute.  We do this by first constructing the
underlying function
\[      
\begin{array}{rcl}
     \u{t} \maps \u{B} &\to& \u{A}  \\
                   x &\mapsto& \u{f}(y)
\end{array}
\]
where $\u{p}(y) = x$.  The respective surjectivity and injectivity of
$\u{p}$ and $\u{i}$ guarantee that $\u{t}$ as just defined is the
unique function making the diagram of underlying sets commute.  To
show that $\u{t}$ induces a map of $\D$ spaces, we need to check that
the following naturality square commutes for every map $d \maps D' \to
D$ in $\D$:
\[
\xymatrix{
B(D) \ar[r]^{t_D}\ar@{-> >}[d]_{B(d)} & A(D) \ar@{ >->}[d]^{A(d)}\\
B(D') \ar[r]_{t_{D'}}   & A(D')\\
}
\]
Since $B$ and $A$ are concrete sheaves, we can check that this diagram
commutes at the level of underlying functions of plots.  Given a plot
$\varphi \in B(D)$, we examine its two images in $A(D')$.  First, we
have
\[      
\begin{array}{rcl}
     \u{A(d)(t_D(\varphi))} \maps \u{D'} &\to& \u{A}  \\
                   c &\mapsto& \u{f}(y)
\end{array}
\]
where $y \in \u{E}$ such that $\u{p}(y) = \u{\varphi}\u{d}(c)$.  The
second image has underlying function defined as follows:
\[      
\begin{array}{rcl}
     \u{t_{D'}(B(d)(\varphi))} \maps \u{D'} &\to& \u{A}  \\
                   c &\mapsto& \u{f}(y')
\end{array}
\]
where $y' \in \u{E}$ such that $\u{p}(y') = \u{\varphi}\u{d}(c)$.  We
are just left to check that given $y,y' \in
\u{p}^{-1}(\u{\varphi}\u{d}(c))$, then $\u{f}(y) = \u{f}(y')$.  This
follows from the commutativity $if = gp$ and that $\u{i}$ is
injective.

Conversely, let $p\maps E\to B$ be a strong epimorphism.  Consider the
concrete presheaf with plots $p_D(E(D))$ for every $D\in\D$.  By the
process of sheafification described in Section \ref{colimits} we obtain 
a $\D\Space$ which we will denote $\tilde{B}$.  We consider the
following commutative diagram:
\[
\xymatrix{
E\ar[r]^{\tilde{p}}\ar@{-> >}[d]_{p} & \tilde{B}\ar@{ >->}[d]\\
B\ar[r]_{1_B}\ar@{-->}[ur]|-{{t}} & B
}
\]
where $\tilde{p}$ has the same underlying function as $p$ and the
unlabeled arrow is the $\D\Space$ map induced by the identity function
on $\u{B}$.  Since $p$ is a strong epimorphism, there exists a unique
$t$ making the diagram commute.  It follows that $\u{t}=1_{\u{B}}$ and
that $B(D)\subseteq\tilde{B}(D)$ for every $D\in\D$.  Every plot
$\varphi \in B(D)$ can then be considered as a plot in $\tilde{B}(D)$,
which arise in two ways.  Either $\varphi$ came from a plot in $E(D)$,
in which case we consider the covering family with just the identity
map $(1 \maps D \to D)$, and there exists a plot $\hat\varphi \in
E(D)$ which maps to $\varphi$, or $\varphi$ arose from sheafification.
In the latter case, there exists a family $(f_i \maps D_i \to D | i
\in I)$ in $\D$ and a compatible collection of plots $\lbrace
\u{p}\u{\varphi_i} \in p_{D_i}(E(D_i)) | i \in I \rbrace$ each of
which is the restriction of $\u{\varphi}$.  We have shown that $p$
makes $B$ a quotient space of $E$.
\end{proof}

\begin{definition}\label{subobject.class}
In a category $\calc$ with finite limits, a {\bf weak subobject classifier} is 
an object $\Omega$ equipped with a morphism $\top\maps 1\to\Omega$ such that, 
given any strong monomorphism $i\maps C' \to C$ in $\calc$, there is a unique 
morphism $\chi_i\maps C\to\Omega$ making
\[
\xymatrix{
C' \ar@{ >->}[r]^{i}\ar[d]_{!} & C\ar[d]^{\chi_i} \\
1\ar@{ >->}[r]_{\top} & \Omega\\
}
\]
a pullback.
\end{definition}

To show $\D\Space$ has a weak subobject classifier for any $\D$ we
need to first show it has finite limits.  It is almost as easy to show it
has arbitrary limits.  For this, we use the
category $\Sh(\D)$ with sheaves on $\D$ as objects and natural
transformations between these as morphisms.  It is well-known
\cite{MM} that this category has
all limits, which can be computed `pointwise': if 
$F \maps C \to \Sh(\D)$ is any diagram of sheaves indexed by a
small category $C$, its limit exists and is given by:
\[ ( \lim F)(D) = \lim \, F(D) .\]

\begin{proposition} \label{limit.prop}
$\D\Space$ has all (small) limits, which may be computed pointwise.
\end{proposition}

\begin{proof}
Let $F\maps C \to \D\Space$ be a diagram.  The limit of the
underlying diagram of sheaves can be computed pointwise; we will show 
that this limit has the concreteness property.  It will follow
that this limit is also the limit in $\D\Space$, since a morphism in 
$\Sh(D)$ between sheaves that happen to be objects of $\D\Space$ is 
automatically a morphism in $\D\Space$.  

For any domain $D$, the diagram $F \maps C \to \D\Space$ gives two 
diagrams of sets, namely its composites with the functors 
\[
\begin{array}{ccl}
\D\Space &\to&      \Set \\
X        &\mapsto&  X(D)  
\end{array}
\]
and
\[
\begin{array}{ccl}
\D\Space &\to&      \Set \\
X        &\mapsto&  \u{X}^{\u{D}} .  
\end{array}
\]
There is a natural transformation between these two diagrams
of sets, namely 
\[   \u{\; \;} \maps F(\alpha)(D) \to \u{F(\alpha)}^{\u{D}} \;. \]
Because each $F(\alpha)$ is a concrete sheaf, each component 
of this natural transformation is one-to-one.
So, taking the limits of both diagrams, we get a one-to-one function
\[       \lim_{\alpha \in C} \, F(\alpha)(D) \to 
         \lim_{\alpha \in C} \, \u{F(\alpha)}^{\u{D}} \]
which by the properties of limits can be reinterpreted 
as a one-to-one function 
\[       \lim_{\alpha \in C} \, F(\alpha)(D) \to 
         (\lim_{\alpha \in C} \u{F(\alpha)})^{\u{D}} \; .\]
Next, since $\u{F(\alpha)} = F(\alpha)(1)$ and limits of sheaves are
computed pointwise, we can reinterpret this as a one-to-one
function
\[       (\lim_{\alpha \in C} F(\alpha))(D) \to 
         (\u{\lim_{\alpha \in C} F(\alpha)})^{\u{D}}  \;.\]
One can check that this function is none other than 
\[    \u{\;\;} \maps (\lim_{\alpha \in C} F(\alpha))(D) \to 
         (\u{\lim_{\alpha \in C} F(\alpha)})^{\u{D}} \; .\]
Since this is one-to-one, the limit of $F$ is a concrete sheaf.
\end{proof}

It follows that the terminal object $1$ is the unique sheaf with exactly 
one plot for each object $D\in\D$.  

\begin{proposition}\label{subobject.prop}
For any concrete site $\D$, the category of $\D$ spaces has a weak subobject 
classifier $\Omega$.
\end{proposition}

\begin{proof}
We define $\Omega$ as in Prop.\ \ref{omega}, and define the map
$\top\maps 1\to\Omega$ as the constant map to $1\in \{0,1\} =
\u{\Omega}$.  Given a strong monomorphism $i\maps A \to X$ of $\D$
spaces we define its characteristic map $\chi_i\maps X\to\Omega$ 
to have the underlying function $\u{\chi}_i$ given by
$\u{\chi}_i(x)=1$ if $x$ is in the image of $\u{i}$ and $\u{\chi}_i(x)=0$
otherwise.  We need to check that this definition makes 
$A$ a pullback.  Consider a $\D$ space $Q$ with maps making this
diagram commute:
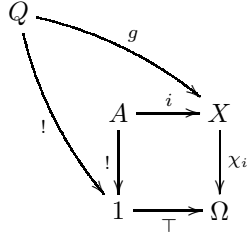
\[
\xymatrix{
  & Q\ar@/_/[ddr]_{!}\ar@/^/[rrd]^{g} & & \\
  & & A \ar[r]^{i}\ar[d]_{!} & X\ar[d]^{\chi_i} & \\
  & & 1\ar[r]_{\top} & \Omega &\\
}
\]
We need to show there exists a unique function $\u{h}\maps \u{Q}\to \u{A}$ 
defining a map of $\D$ spaces $h\maps Q\to A$.   Since the outer edges of the 
diagram commute we can define $\u{h}$ such that $\u{m}\u{h}=\u{g}$.  Since $A$ 
is a subspace of $X$ it is clear that $h$ is a map of 
$\D$ spaces.
\end{proof}

\subsection{Parametrized Mapping Spaces}

We next turn to the existence of parametrized mapping spaces between
$\D$ spaces over a fixed base $B$.

\begin{definition}\label{slice.cat}
Given an object $B$ in a category $\calc$, the category of
{\bf objects over $B$} (sometimes called the {\bf slice category}
of $B$), has morphisms $f\maps E\to B$ in $\calc$ as objects and 
commuting triangles
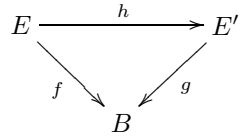
\[
\xymatrix{
E\ar[rr]^{h}\ar[dr]_{f} && E^\prime\ar[dl]^{g}\\
& B &
}
\]
as morphisms.  We denote this category by $\calc_B$.
\end{definition}

We can think of these as `bundles' over $B$, in a very general
sense, not assuming any sort of local triviality.  The product in the 
category of objects over $B$ is given by the pullback:
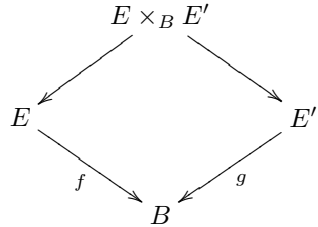
\[
\xymatrix{
& E\times_B E^\prime\ar[dl]\ar[dr] & \\
E\ar[dr]_{f} & & E^\prime\ar[dl]^{g}\\
& B &\\
}
\]

\begin{definition}\label{cart.closed}
A category $\calc$ is {\bf cartesian closed} if it has finite products
and for every object $Y \in \calc$, the functor 
\[  - \times Y \maps \calc \to \calc \]
has a right adjoint, called the {\bf internal hom} and denoted
\[  \calc(Y, -) \maps \calc \to \calc  .\]
\end{definition}

\noindent 
The fact that $\calc(X, -)$ is right adjoint to $- \times X$ means
that we have a natural bijection of sets
\[     \hom(X \times Y, Z) \cong \hom(X, \calc(Y,Z))  \]
but a standard argument shows that we also have a natural isomorphism 
in $\calc$:
\[     \calc(X \times Y, Z) \cong \calc(X, \calc(Y,Z)) .  \]

\begin{definition}\label{loc.cart.closed}
A category $\calc$ is called {\bf locally cartesian closed} if for every 
$B\in\calc$, the category $\calc_B$ of objects over $B$ is cartesian closed.
Given objects $X,Y$ over $B$, we call the internal hom $\calc_B(X,Y)$ a 
{\bf parametrized mapping space}.
\end{definition}

We want to show that the category of $\D$ spaces is locally cartesian closed.  
To do this we need to determine the product and internal $\hom$ in the 
category of $\D$ spaces over some $\D$ space $B$.  Given two spaces over $B$,
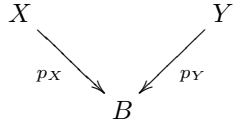
\[
\xymatrix{
X\ar[dr]_{p_X} && Y\ar[dl]^{p_Y}\\
& B &
}
\]
the product is given by the pullback $X\times_B Y$ in the category of $\D$ 
spaces.  It is easily checked that the universal property holds by considering 
the universal property of the pullback.  Alternatively, the $\D$ space 
structure on $X\times_B Y$ can be quickly obtained by the following lemma:

\begin{lemma}
The monomorphism $m\maps X\times_A Y\to X\times Y$ in $\D\Space$ given by 
inclusion of sets $\u{X\times_A Y}\hookrightarrow \u{X\times Y}$ is a 
strong monomorphism.
\end{lemma}

\begin{proof}
Given $C\in\D$, then for any $\varphi\in(X\times Y)(C)$ such that 
$\u{\varphi}(\u{C})\subseteq\u{X\times_A Y}$, the following diagram commutes:
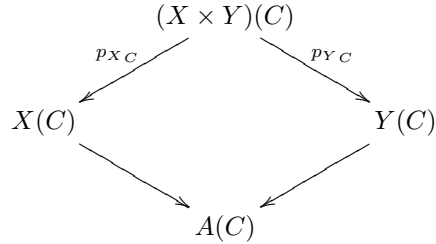
\[\xymatrix{
& (X\times Y)(C)\ar[dl]_{{p_X}_C}\ar[dr]^{{p_Y}_C} &\\
X(C)\ar[dr] & & Y(C)\ar[dl]\\
& A(C) &
}
\]
and thus $\varphi\in(X\times_A Y)(C)\subseteq(X\times Y)(C)$.
\end{proof}

\begin{proposition}\label{loc.cart.prop}
For any concrete site $\D$, the category of $\D$ spaces is locally 
cartesian closed.
\end{proposition}

\begin{proof}
To keep our notation terse, we write $\cald$ for $\D\Space$ in what
follows.  So, given a $\D$ space $B$, $\cald_B$ stands for the 
category of $\D$ spaces over $B$, and given $X,Y \in \cald_B$ we will
denote their parametrized mapping space by $\cald_B(X,Y)$.

Of course, we need to prove that this internal hom exists.
To describe it, we start by describing its underlying set,
which we denote by $\u{\cald}_B(X,Y)$.
First, suppose $p \maps X \to B$ is any $\D$ space over $B$.  
Given a point $b \in \u{B}$, let the {\bf fiber of $X$ 
over $b$}, denoted $X_b$, be the set 
\[         \u{p}^{-1}(b) \subseteq \u{X} \]
with the unique $\D$ space structure for which the inclusion $i \maps X_b \to X$
makes $X_b$ into a subspace of $X$.  
Then, we have 
\[         \u{\cald}_B(X,Y) = \coprod_{b \in \u{B}} \u{\cald}(X_b,Y_b) ,\]
where $\u{\cald}(X_b,Y_b)$ is the underlying set of $\cald(X_b,Y_b)$. 
This set sits over the set $\u{B}$ in an obvious way, which we denote as
\[         q \maps \u{\cald}_B(X,Y) \to \u{B}  .\]

Next we describe the plots for $\cald_B(X,Y)$.  Given $C \in \D$ and
function $\u{\varphi}\maps \u{C}\to \u{\cald}_B(X,Y)$, we need to say
when this is the function underlying a plot with domain $C$.  We can
consider $C$ as a representable $\D$ space (see Prop.\
\ref{representable.prop}).  By composing $\u{\varphi}$ with $q$ we see
that $\u{C}$ is a set over $\u{B}$.  We say that $\u{\varphi}$
determines a plot for $\cald_B(X,Y)$ if this composite function
$q\u{\varphi}$ underlies a map of $\D$ spaces and
\[
\u{C}\times_{\u{B}} \u{X}\stackrel{\u{\varphi}\times_{\u B} 1}\longrightarrow 
\coprod_{b\in \u{B}} \u{\cald}(X_b,Y_b)\times_{\u B} \u{X}\stackrel{\u{\ev}}
\longrightarrow \u{Y}
\]
underlies a map of $\D$ spaces, where $\u{\ev}$ is the evaluation map.  

We need to check that these plots for $\cald_B(X,Y)$, indeed describe
a sheaf.  Let $(f_i \maps C_i \to C | i \in I)$ be a covering family
with a compatible collection of plots $\lbrace \u{\varphi_i} \maps
\u{C_i} \to \u{C} \rbrace$.  From this collection we can define a
function $\u{\varphi} \maps \u{C} \to \u{\cald}_B(X,Y)$.  It is clear
since $B$ is a sheaf, that $\varphi$ satisfies the first condition of
being a plot for $\cald_B(X,Y)$.  We then just need to check that

\[
\u{C} \times_{\u{B}} \u{X} \stackrel{\u{\varphi} \times_{\u B} 1}
\longrightarrow \u{\cald}_B(X,Y) \times_{\u B} \u{X}
\stackrel{\u{\ev}}\longrightarrow \u{Y} 
\]
underlies a map of smooth spaces.

Let $\psi \maps C' \to C \times_B X$ be a plot.
Projecting out of the first coordinate, we obtain a map $\pi_1\psi
\maps C' \to C$ in $\D$.  Recalling the property in the definition of
a covering family, corresponding to the family covering $C$ and the
map $\pi_1\psi \maps C' \to C$, there exists a covering family $(g_j
\maps D_j \to C' | j \in J)$ with the following property: for each $j
\in J$, there exists a map $h_{ij} \maps D_j \to C_i$ for some $i \in
I$, such that the following diagram commutes:
\[
\xymatrix{
C' \ar[r]^{\psi} & C \times_B X \ar[r]^{\pi_1} & C \\
D_j \ar[u]^{g_j} \ar[rr]_{h_{ij}} & & C_i \ar[u]_{f_i}
}
\]
We define for each $j \in J$, a function $\u{\tau_{ij}} \maps \u{D_j}
\to \u{C_i} \times_{\u{B}} \u{X}$ by $\u{\tau_{ij}}(d) =
(\u{h_{ij}}(d),\u{\pi_2\psi g_j}(d))$.  These functions
$\u{\tau_{ij}}$ will be plots if the projection to each component is a
plot.  This is clear since $h_{ij}$ is a map in $\D$ and $\psi g_j$ is
a plot in $C \times_B X$.  We consider the following diagram for each
$j \in J$:
\[
\xymatrix{
D_j \ar[r]^{g_j} \ar[drr]_{\tau_{ij}} & C' \ar[r]^{\psi} & C \times_B
                 X \ar[r]^<<<<<{\varphi \times_B 1} & \cald_B(X,Y)
                 \times_B X \ar[r]^>>>>>>{\ev} & Y \\ & & C_i \times_B X
                 \ar[ur]_{\varphi_i \times_B 1} & & }
\]
It is easy to check that this diagram commutes since for $d \in \u{D}$, 
\[
 (\u{\varphi_i h_{ij}}(d),\u{\pi_2\psi g_j}(d)) = 
(\u{\varphi\pi_1\psi  g_j}(d),\u{\pi_2\psi g_j}(d)), 
\]
which follows from $\varphi\pi_1\psi g_j = \varphi f_i h_{ij} =
\varphi_i h_{ij}$.  It follows that for each $j \in J$, that
\[
 D_j \stackrel{g_j}\longrightarrow C' \stackrel{\psi}\longrightarrow C
 \times_B X \stackrel{\varphi \times_B 1}\longrightarrow
 \cald_B(X,Y) \times_B X \stackrel{\ev}\longrightarrow Y 
\]
is a plot.

We now check that $\lbrace \ev(\varphi \times_B 1_X)\psi g_j \in Y(D_j)
\rbrace_{j\in J}$ is a compatible collection of plots corresponding to
the covering family $(g_j \maps D_j \to C' | j \in J)$.  Let $d \in
\u{g_j}(\u{D_j}) \cap \u{g_k}(\u{D_k})$ with $e_j \in \u{D_j}$, $e_k
\in \u{D_k}$ such that $\u{g_j}(e_j) = d = \u{g_k}(e_k)$.  We have
\[
 \u{\ev(\varphi \times_B 1_X)\psi g_j}(e_j) = \u{(\varphi_i \times_B
 1_X)\tau_{ij}}(e_j) = (\u{\varphi_i h_{ij}}(e_j),\u{\pi_2\psi
 g_j}(e_j)) 
\]
and
\[
 \u{\ev(\varphi \times_B 1_X)\psi g_k}(e_k) = \u{(\varphi_l \times_B
 1_X)\tau_{lk}}(e_k) = (\u{\varphi_l h_{lk}}(e_k),\u{\pi_2\psi
 g_k}(e_k)). 
\]
We need to show equality of the rightmost terms.  The second
components are clearly equal.  The equality of the first components
follows from $\u{f_i h_{ij}}(e_j) = \u{f_l h_{lk}}(e_k)$, and that the
family of plots $\lbrace \varphi_i \in \cald_B(X,Y)(C_i) | i \in I
\rbrace$ is compatible.  Since $Y$ is a smooth space, we have by the
sheaf condition that $\ev(\varphi \times_B 1_X)\psi$ is a plot.  We
have now shown that
\[
 C \times_B X \stackrel{\varphi \times_B 1}\longrightarrow
 \cald_B(X,Y) \times_B X \stackrel{\ev}\longrightarrow Y 
\]
is a smooth map.  It follows that $\cald_B(X,Y)$ is a sheaf.

Given $\D$ spaces $X$,$Y$,$Z$ over $B$ with projections
$p_X$,$p_Y$, and $p_Z$ respectively, we define a bijective
correspondence between functions of the form
\[f\maps \u{Z\times_B X}\to \u{Y}\] 
and
\[\tilde{f}\maps \u{Z}\to \coprod_{b\in \u{B}}\u{\cald}(X_b,Y_b).\]
Given $f$ we define $\tilde{f}$ in the following way: given $z\in
\u{Z}$ with $\u{p_Z}(z)=b\in \u{B}$, if $X_b$ is
empty, then $\tilde{f}(z)$ is defined to be the unique map $!\maps
\emptyset\to Y_b$.  Otherwise, $\tilde{f}(z):=f(z,-)\maps
X_b\to Y_b$.  Starting with a map
$\tilde{f}\maps \u{Z}\to \coprod_{b\in
\u{B}}\u{\cald}(X_b,Y_b)$, the map $f\maps
\u{Z} \times_{\u{B}} \u{X}\to \u{Y}$ is defined by $f(z,x):= \tilde{f}(z)(x)$.

Next we must show that $f$ defines a map of $\D$ spaces if and only if 
$\tilde{f}$ defines a map of $\D$ spaces, and that this correspondence 
is natural.

First, let us show that if we have a map of $\D$ spaces $f\maps
Z\times_B X\to Y$, then the function $\tilde{\u{f}}\maps \u{Z}\to
\coprod_{b\in \u{B}}\u{\cald}(X_b,Y_b)$
constructed above determines a map of $\D$ spaces.  Given an object
$C\in\D$ and a plot of $Z(C)$, $\u{\varphi}\maps \u{C}\to
\u{Z}$, we treat $C$ as a $\D$ space over $B$ and obtain a function 
$\u{\varphi}\times_{\u{B}} 1 \maps \u{C}\times_{\u{B}} \u{X} \to 
\u{Z}\times_{\u{B}} \u{X}$.  Since $\u{\varphi}$ is a plot, this
determines a map of $\D$ spaces.  One can check that the following
diagram commutes and it follows that $\tilde{f}$ determines a
$\D\Space$ map:
\[
 \def\objectstyle{\scriptstyle}
  \def\labelstyle{\scriptstyle}
   \xy
 (-40,0)*{\u{C}\times_{\u{B}} \u{X}}="1"; 
 (-15,0)*{\u{Z}\times_{\u{B}} \u{X}}="2";
 (15,0)*{\coprod\u{\cald}(X_b,Y_b)\times_{\u{B}} \u{X}}="3"; 
 (40,0)*{\u{Y}}="4";
{\ar^{\u{\varphi}\times_{\u{B}} 1}  "1";"2"};
{\ar^>>>>>>>>>{\u{\tilde{f}}\times_{\u{B}} 1} "2";"3"};
{\ar^>>>>>>>>{\u{\ev}} "3";"4"};
{\ar@/_2pc/_{\u{f}}  "2";"4"};
\endxy
\\ \\
\]
Conversely, given a $\D\Space$ map $\tilde{f}\maps Z\to
\coprod_{b\in \u{B}}\cald(X_b,Y_b)$ and a plot $\u{\varphi}\maps
\u{C}\to\u{Z} \times_{\u{B}} \u{X}$, the composite along the top of the
following diagram is a plot of $Y$ and the commutativity of the
diagram implies that the function $f\maps \u{Z} \times_{\u{B}} \u{X}\to \u{Y}$
defines a $\D\Space$ map:
$$ \def\objectstyle{\scriptstyle}
  \def\labelstyle{\scriptstyle}
   \xy
 (-40,0)*{\u{C}}="1"; 
 (-15,0)*{\u{Z} \times_{\u{B}} \u{X}}="2";
 (15,0)*{\coprod\u{\cald}(X_b,Y_b)\times_{\u{B}} \u{X}}="3"; 
 (40,0)*{\u{Y}}="4";
{\ar^{\u{\varphi}}  "1";"2"};
{\ar^>>>>>>>>>{\tilde{\u{f}}\times_{\u{B}} 1} "2";"3"};
{\ar^>>>>>>>>{\u{\ev}} "3";"4"};
{\ar@/_2pc/_{f}  "2";"4"};
\endxy
\\ \\
$$
To check naturality in the $Z$ variable, we consider a $\D\Space$ map 
$g\maps Z'\to Z$ and ask that the following diagram commutes.  We note that 
we only need to check the commutativity of the functions of the underlying sets.
\[
\xymatrix{
\cald(Z,\coprod \cald(X_b,Y_b))\ar[r]\ar[d] & \cald(Z\times_B X,Y)\ar[d]\\
\cald(Z',\coprod \cald(X_b,Y_b))\ar[r] & \cald(Z'\times_B X,Y)
}
\]
Let $f\maps Z\to \coprod_{b\in \u{B}}\cald(X_b,Y_b)$ 
be an element in the top left corner.  Following the diagram down 
and right, we first obtain a map $fg\maps Z'\to\coprod_{b\in
\u{B}}\cald(X_b,Y_b)$ and then a map $\widetilde{fg}\maps Z'\times_B X\to Y$.  
Given $(z',x)\in Z'\times_B X$, we see that $\widetilde{fg}(z',x)=\tilde{f}(g(z'),x)$.  

Following the diagram the other way, $f$ is first taken to $\tilde{f}$,
and then $\tilde{f}$ is taken to a map from $Z'\times_B X$ to $Y$ by the
pullback of $Z'$ and $X$.  This induces a map $Z'\times_B X$ to
$Z\times_B X$ which we compose with $\tilde{f}$ to obtain the desired
map.  It follows that $(z',x)\in Z'\times_B X\mapsto (g(z'),x)\in
Z\times_B X\mapsto \tilde{f}(g(z'),x)\in Y$.  Thus the diagram
commutes.  Given a $\D\Space$ map $h\maps Y \to Y'$ the commutativity
of the following diagram is given by composing at each step with $h$
in the appropriate manner:

\[
\xymatrix{
\cald(Z,\coprod \cald(X_b,Y_b))\ar[r]\ar[d] & \cald(Z\times_B X,Y)\ar[d]\\
\cald(Z,\coprod \cald(X_b,{Y'}_b))\ar[r] & \cald(Z\times_B X,Y')
}
\]
It follows that the correspondence is natural in $Y$, and thus that
$\D\Space$ is locally cartesian closed.
\end{proof}

\subsection{Colimits}
\label{colimits}

In Prop.\ \ref{limit.prop} we showed that the category of $\D$ spaces has 
limits, which can be computed pointwise.  To compute colimits in $\D\Space$, 
we need some facts about `sheafification' and also `concretization'. 

Given any site $\D$, sheafification is a functor 
that takes presheaves on $\D$ to sheaves on $\D$, but does not affect 
presheaves that are already sheaves \cite{MM}.  We denote this functor by 
\[       S\maps \Set^{\D^\op}\to \Sh(\D)  ,\]
where $\Set^{\D^\op}$ is the category of presheaves on $\D$ and $\Sh(\D)$ 
is the category of sheaves on $\D$.  (In both these categories, the morphisms
are just natural transformations.)  The functor $S$ is left adjoint to 
the inclusion
\[        I \maps \Sh(\D)\to\Set^{\D^\op}  .\]
Since $S$ is a left adjoint, it preserves colimits.  So, to compute a
colimit of sheaves we can compute the colimit of
their underlying presheaves as objects in $\Set^{\D^\op}$
and then sheafify the result.

Grothendieck's `plus construction' \cite{MM} gives an explicit
recipe for sheafification.  Given a presheaf $X$, the plus 
construction gives a new presheaf $X^+$ by taking the colimit over 
covering families of compatible collections for those families:
\[X^+(C) = {{\rm colim}}_{R\in\mathcal{J}(C)}\, \mathcal{F}(R,X)\]
where $\mathcal{F}(R,X)$ is the set of compatible collections for the
covering family $R$ of $C \in \D$.  Applying the plus construction to
any presheaf gives a {\bf separated} presheaf, which is like a sheaf except
that the existence property in Def.\ \ref{sheaf} is dropped, and only 
uniqueness is required.  Applying the plus construction to any separated
presheaf gives a sheaf.  So, if $X$ is a presheaf, $X^{++}$ is a sheaf ---
and in fact it is the sheafification of $X$.

Next we turn to concretization, which makes presheaves
`concrete': 

\begin{definition}  Given a concrete site $\D$, we say
a presheaf $X \maps \D^\op \to \Set$ is {\bf concrete} if for every
object $D \in \D$, the function sending plots $\varphi \in X(D)$ to
functions $\u{\varphi} \maps \hom(1,D) \to X(1)$ is one-to-one.  We
denote the category of concrete presheaves on $\D$ and natural
transformations between these by $\Conc(\Set^{\D^\op})$.
\end{definition}

\noindent
For any presheaf $X$ on $\D$ there is a
concrete presheaf $L(X)$ for which $L(X)(D)$ consists of equivalence
classes of plots $\varphi \in X(D)$, where $\varphi \sim \varphi'$ if
and only if $\u{\varphi} = \u{\varphi}'$.  Since these equivalence
classes can be identified with functions $\u{D}\to \u{X}$, the image 
of $L$ on a morphism $f\maps X\to Y$ is completely determined by the function 
$\u{f}\maps \u{X}\to \u{Y}$.  It follows that $L$ preserves identities 
and composition.  So, we obtain a functor called \textbf{concretization}:
\[  L \maps \Set^{\D^\op} \to \Conc(\Set^{\D^\op})  .\]
On the other hand, there is an obvious inclusion functor 
\[   R\maps \Conc(\Set^{\D^\op}) \to \Set^{\D^\op} .\]

\begin{lemma}\label{adjoint.lem}
$L$ is left adjoint to $R$.
\end{lemma}

\begin{proof}
Given a presheaf $X$, a concrete presheaf $Y$ and a natural
transformation $f\maps L(X)\to Y$, we obtain a natural transformation
$\tilde{f}\maps X\to R(Y)$ pointwise as $\tilde{f}_D\maps X(D)\to
R(Y)(D)$ defined by
\[\varphi\mapsto f_D([\varphi]),\]
where we think of $f_D([\varphi])$ in $Y(D)$ as a plot of $R(Y)(D)$ under the 
inclusion.  Conversely, given a natural transformation $f\maps X\to R(Y)$ we define 
$\tilde{f}_D\maps L(X)(D)\to Y(D)$ pointwise by
\[[\varphi]\mapsto f_D(\varphi),\]
which is well defined since the equivalence relation is defined by
underlying functions.  This defines a bijective correspondence.  To check
naturality in the first argument, it is sufficient to show that given
$h\maps X\to X'$ and $g\maps L(X')\to Y$ that the
following square commutes for every $D\in\D$.
\[
\xymatrix{
\hom(L(X')(D),Y(D))\ar[r]\ar[d] & \hom(X'(D),R(Y)(D))\ar[d]\\
\hom(L(X)(D),Y(D))\ar[r] & \hom(X(D),R(Y)(D))
}
\]
Along the top and right, $g_D$ gets sent to a map $(\varphi\mapsto
g_D([h_D([\varphi])))$ and along the left and bottom to a map $(\varphi\mapsto
g_D(L(h)_D([\varphi])))$.  These are equal since maps between presheaves
preserve the equivalence class.  Naturality in the second argument
follows similarly.
\end{proof}

\begin{lemma}
Any concrete presheaf on a concrete site is a separated presheaf.
\end{lemma}

\begin{proof}
Clear.  
\end{proof}

It follows that for any concrete presheaf, sheafification 
is the same as {\it one} application of Grothendieck's plus 
construction.  This brings us to the following lemma:

\begin{lemma} \label{plus.construction}
Given a concrete presheaf $X$, $X^+$ is a concrete sheaf.  
\end{lemma}

\begin{proof}
In the interest of presenting a simple argument, we now replace
the coverage on our concrete site $\D$ by the
Grothendieck topology $\mathcal{J}$ which has the same sheaves.

$X^+(D)$ is the colimit of a diagram (indexed by the sieves in
$\mathcal{J}(D)$) of sets of compatible families of plots in $X(D)$.
Given sieves $R$ and $R'$ indexed by the sets $I$ and $J$,
respectively, there is a function in the colimit diagram from the set
of compatible families for a sieve $R$ to the set of compatible
families for the sieve $R'$ exactly when $R' \subseteq R$.  This
function takes a compatible family $\lbrace \varphi_i \rbrace$ to the
compatible family $\lbrace \varphi_{j} \rbrace \subseteq \lbrace
\varphi_i \rbrace$ for $R'$.

Since the sheafification process is pointwise a colimit of sets, given
plots $\varphi,\varphi' \in X^+(D)$ with the same underlying functions
$\u{\varphi} = \u{\varphi'}$, there must be compatible families in the
diagram which are mapped to each $\varphi$ and $\varphi'$.  Since
given a morphism $f_i \maps C \to D$ in $R$ we have $X^+(f)(\varphi) =
\varphi_i$ and $X^+(f)(\varphi') = {\varphi'}_i$, it follows that if
the two plots are in the image of compatible families $\lbrace
\varphi_i \rbrace_{i\in I}$ and $\lbrace {\varphi'}_i \rbrace_{i\in
I}$ for the same sieve $R \in \mathcal{J}(D)$ with indexing set $I$,
then we have $\lbrace \varphi_i \rbrace_{i\in I} = \lbrace
{\varphi'}_i \rbrace_{i\in I}$.  Hence $\varphi = \varphi'$.

If $\varphi$ and $\varphi'$ are in the image of families for sieves
$R$ and $R'$ respectively, then we can show that each of these
functions factors through a set of compatible families for the common
refinement $R \cap R'$.  Since the sieve $R \cap R'$ is a subset of
each $R$ and $R'$, the existence of the functions factoring through
this set is guaranteed as long as $R \cap R'$ is in $\mathcal{J}(D)$.
That the intersection of two covering sieves is a covering sieve follows 
directly from the axioms of a
Grothendieck topology.  Since we have focused almost entirely on
coverages, we refer the the reader to Mac Lane and Moerdijk \cite{MM}
for explanation and proof.  Now the preimages of $\varphi$ and
$\varphi'$ will be sent to the same compatible family for $R \cap R'$
and thus $\varphi = \varphi'$.
\end{proof}

\begin{proposition} \label{colimit.prop}
The category $\D\Space$ has all (small) colimits.
\end{proposition}

\begin{proof}
Given a diagram of $\D$ spaces $F \maps A \to\D\Space$, 
let $\tilde{F} \maps A \to \Set^{\D^\op}$ be the underlying
diagram of presheaves.  We can compute
the colimit $P$ of $\tilde F$ pointwise.  Given any presheaf we can 
concretize and then sheafify to obtain a $\D$ space.  Since each of these 
functors is a left adjoint, this entire process preserves colimits.  Also, 
if the presheaf is already a $\D$ space then the process will have no 
effect.  So we can apply this process to $P$ and $\tilde{F}$, and it 
follows that the $\D$ space obtained from the presheaf $P$ is the colimit 
of $F$ in $\D\Space$.
\end{proof} 

It is interesting to note that in two of his papers,
Chen called spaces satisfying axioms 1 and 3 but not necessarily axiom 2 
`predifferentiable' spaces \cite{Chen:1975,Chen:1977}.  These are the same 
as concrete presheaves on $\Chen$.
Chen described a systematic process for improving any predifferentiable
space to a Chen space.  This process is just the plus construction!
The point is that by Lemma \ref{plus.construction}, we can turn
a concrete presheaf into a concrete sheaf using
the plus construction.  

The following result is an easy spinoff of what we have done:

\begin{proposition}
Every $\D$ space is a colimit of representable $\D$ spaces.
\end{proposition}

\begin{proof}
It is well known that any presheaf is the colimit of representable
presheaves \cite{MM}.  So, given a $\D$ space $X$, there is a diagram of 
representables having the underlying presheaf of $X$ as 
its colimit in $\Set^{\D^\op}$.  As in Prop.\  \ref{colimit.prop}, 
applying the concretization functor $L$ and then the sheafification 
functor $S$, we can send this diagram into $\D\Space$ 
while preserving colimits.  Since each representable is a $\D$ space by Prop.\ 
\ref{representable.prop} and $X$ was chosen to be a $\D$ space, we obtain a diagram
exhibiting $X$ as a colimit of representables in $\D\Space$.  
\end{proof}

Most of our results on generalized spaces can be summarized in this theorem:

\begin{theorem}\label{quasitopos.thm}
For any concrete site $\D$, the category of $\D$ spaces is a quasitopos 
with all (small) limits and colimits.   
\end{theorem}

\begin{proof}
Recall that a `quasitopos' is a locally cartesian closed category
with finite colimits and a weak subobject classifier.  We showed that
$D\Space$ has all limits in Prop.\ \ref{limit.prop}, that it has
weak subobject classifier in Prop.\ \ref{subobject.prop},
that it is locally cartesian closed in Prop.\ \ref{loc.cart.prop},
and that it has all colimits in Prop.\ \ref{colimit.prop}.
\end{proof}

\subsection*{Acknowledgements}

This work had its origin in collaborations between the first author,
Urs Schreiber \cite{BaezSchreiber:2004,BaezSchreiber:2005} and Toby
Bartels \cite{Bartels:2006}. We thank James Dolan for invaluable help,
such as explaining the notion of `concrete sheaf' that we use here,
and pointing out the example of simplicial complexes.  We also thank
$n$-Category Caf\'e regulars including Bruce Bartlett, Todd Trimble
and especially Andrew Stacey for online discussions, and Dan
Christensen and Chris Rogers for many helpful conversations.

\end{document}